\documentclass[10pt]{article}

\usepackage[utf8]{inputenc}

\usepackage{mathtools}

\usepackage{amsfonts}
\usepackage{amssymb}
\usepackage{amsthm}
\usepackage[date=year,url=false,giveninits=true,isbn=false,maxbibnames=99,sortcites]{biblatex}
\usepackage{bm}
\usepackage{booktabs}
\usepackage{caption}
\usepackage{centernot}
\usepackage{enumerate}
\usepackage{float}
\usepackage[a4paper]{geometry}
\usepackage{graphicx}
\usepackage{lmodern}
\usepackage{microtype}
\usepackage{multirow}
\usepackage{tikz}
\usepackage{xcolor}

\PassOptionsToPackage{hyphens}{url}

\usetikzlibrary{calc}
\usetikzlibrary{cd}
\usetikzlibrary{matrix}

\tikzset{ shorten <>/.style={ shorten >=#1, shorten <=#1}}

\restylefloat{table}

\usepackage[colorlinks=true,linkcolor=red!50!black,citecolor=green!50!black,urlcolor=blue!50!black,hyperindex,breaklinks]{hyperref}	
\usepackage{url}
\usepackage[capitalise]{cleveref}

\renewbibmacro{in:}{}
\DeclareFieldFormat[article]{title}{#1} 
\newtheoremstyle{break}
	{}
	{}
	{\itshape}
	{}
	{\bfseries}
	{}
	{\newline}
	{}
\newtheoremstyle{romanfont}
	{}
	{}
	{}
	{}
	{\bfseries}
	{}
	{\newline}
	{}
\newtheoremstyle{cite}
	{}
	{}
	{\itshape}
	{}
	{\bfseries}
	{}
	{\newline}
	{\thmname{#1} \thmnumber{#2} \thmnote{#3}}
\newtheoremstyle{romancite}
	{}
	{}
	{}
	{}
	{\bfseries}
	{}
	{\newline}
	{\thmname{#1} \thmnumber{#2} \thmnote{#3}}
\theoremstyle{break}
\newtheorem{thm}[subsection]{Theorem}
\newtheorem{cor}[subsection]{Corollary}
\newtheorem{lemma}[subsection]{Lemma}
\newtheorem{propn}[subsection]{Proposition}
\newtheorem{introthm}{Theorem}
\newtheorem{introcor}{Corollary}

\theoremstyle{romanfont}
\newtheorem{defn}[subsection]{Definition}
\newtheorem*{notation*}{Notation}

\theoremstyle{cite}
\newtheorem{thmcite}[subsection]{Theorem}
\newtheorem{lemmacite}[subsection]{Lemma}
\newtheorem{propncite}[subsection]{Proposition}

\theoremstyle{romancite}
\newtheorem{defncite}[subsection]{Definition}

\crefname{defn}{Definition}{Definitions}
\crefname{propn}{Proposition}{Propositions}
\crefname{cor}{Corollary}{Corollaries}
\crefname{thm}{Theorem}{Theorems}

\let\H\relax

\renewcommand{\epsilon}{\varepsilon}

\newcommand{\bluehigh}[1]{{\color{blue} #1}}
\newcommand{\GHom}{\Hom_G}
\newcommand{\inj}{\hookrightarrow}
\newcommand{\bigmoduleshape}[2]{\begin{array}{c}	#1 \\ #2	\end{array}}
\newcommand{\threemoduleshape}[3]{\begin{array}{c}	#1 \\	#2 \\	#3 \end{array}}
\newcommand{\nin}{\notin}
\newcommand{\norm}{\trianglelefteq}

\newcommand{\perm}{\mathcal{L}}
\newcommand{\permG}{\mathcal{L}/\mathcal{L}^G}
\newcommand{\permk}{\perm/k}
\newcommand{\redhigh}[1]{{\color{red} #1 }}
\newcommand{\surj}{\twoheadrightarrow}

\newcommand{\upth}{\textsuperscript{th}}

\DeclareMathOperator{\Char}{char}
\DeclareMathOperator{\E}{E}
\DeclareMathOperator{\Ext}{Ext}

\DeclareMathOperator{\H}{H}
\DeclareMathOperator{\head}{head}
\DeclareMathOperator{\heart}{\mathcal{H}}
\DeclareMathOperator{\Hom}{Hom}
\DeclareMathOperator{\Ind}{Ind}
\DeclareMathOperator{\Irr}{Irr}

\DeclareMathOperator{\PC}{\mathcal{P}}
\DeclareMathOperator{\PGL}{PGL}

\DeclareMathOperator{\PSL}{PSL}

\DeclareMathOperator{\rad}{rad}

\DeclareMathOperator{\SL}{SL}

\DeclareMathOperator{\soc}{soc}

\DeclareMathOperator{\Syl}{Syl}

\DeclarePairedDelimiter{\abs}{\lvert}{\rvert}
\DeclarePairedDelimiter{\ceil}{\lceil}{\rceil}
\DeclarePairedDelimiter{\floor}{\lfloor}{\rfloor}

\begin{document}

\begin{center}
{\Large Cohomology of \texorpdfstring{\(\PSL_2(q)\)}{PSL(2,q)}}

{\large Jack Saunders\footnotemark} \footnotetext{This work was partially done during my PhD, thus I would like to thank Corneliu Hoffman \& Chris Parker for everything they have done for me so far. I am also grateful to the LMS for their financial support via the grant ECF-1920-30 and to Gunter Malle for his support during the corresponding fellowship.}

Department of Mathematics and Statistics, The University of Western Australia\\
35 Stirling Highway, Perth, WA 6009, Australia\vskip-0.8em
\href{mailto:jack.saunders@uwa.edu.au}{jack.saunders@uwa.edu.au}
\end{center}

\begin{abstract}
	In 2011, Guralnick and Tiep proved that if \(G\) was a Chevalley group with Borel subgroup \(B\) and \(V\) an irreducible \(G\)-module in cross characteristic with \(V^B = 0\), then the the dimension of \(\H^1(G,V)\) is determined by the structure of the permutation module on the cosets of \(B\). We generalise this theorem to higher cohomology and an arbitrary finite group, so that if \(H \leq G\) such that \(O_{r'}(H) = O^r(H)\) and \(V^H = 0\) for \(V\) a \(G\)-module in characteristic \(r\) then \(\dim \H^1(G,V)\) is determined by the structure of the permutation module on cosets of \(H\), and \(\H^n(G,V)\) by \(\Ext_G^{n-1}(V^*,M)\) for some \(kG\)-module \(M\) dependent on \(H\). We also determine \(\Ext_G^n(V,W)\) for all irreducible \(kG\)-modules \(V\), \(W\) for \(G \in \{\PSL_2(q), \PGL_2(q), \SL_2(q)\}\) in cross characteristic.
\end{abstract}

\section{Introduction}

Group cohomology is intricately tied with the extension theory of finite groups. If \(N\) is an abelian group with fixed action of another group \(G\) on \(N\) by automorphisms, then \(\H^1(G,N)\) parameterises the conjugacy classes of complements to \(N\) in \(N \rtimes G\), and \(\abs{\H^2(G,N)}\) is the number of equivalence classes of extensions of \(G\) by \(N\). The first cohomology group \(\H^1(G,N)\) is also linked to generating sets for groups and their modules \cite{GuralnickHoff,AschbacherGuralnickApplications}, and \(\H^2(G,V)\) for \(V\) an irreducible \(G\)-module gives information about the number of relations needed in profinite presentations of \(G\) \cite{LubotzkyPresentations,GuralnickPresentations}.

Due to these various applications, it is natural to want to bound the cohomology of some group \(G\) and some \(G\)-module \(V\), and in particular it is natural to start with the case where \(G\) is a finite simple group and \(V\) is an irreducible \(kG\)-module for \(k\) an algebraically closed field of characteristic \(r\). The cohomology in this case has been investigated extensively, though still remarkably little is known. In 1986, Guralnick conjectured \cite{GuralnickDimension86} that if \(V\) is a faithful irreducible \(kG\)-module then there exists some absolute constant \(c\) such that \(\dim \H^1(G,V) \leq c\), and originally it was conjectured that \(c = 2\). Since then, examples were found in 2003 and 2008 for groups of Lie type \(G(p^n)\) in defining characteristic (\(r = p\)) \cite{BrayWilson,ScottCounter} with 3-dimensional cohomology, and more recently in 2020 L{\"u}beck \cite[Theorem 4.7]{LubeckComputation} found an irreducible module \(V\) in defining characteristic for \(\operatorname{E}_6(q)\) with \(\dim \H^1(G,V) = 3537142\) (a similar computation of L{\"u}beck in 2012 was also used to disprove Wall's conjecture \cite{WallConjectureCounterexample}). Thus the Guralnick conjecture is likely false, but it is still important to study the behaviour of group cohomology.

Let \(G\) be a rank \(e\) finite group of Lie type, defined in characteristic \(p\) with Weyl group \(W\). If \(r \neq p\), then the current best bounds for irreducible \(V\) are \(\dim \H^1(G,V) \leq \abs{W}^{\frac{1}{2}} + e - 1\) from \cite{GuralnickTiep2} in 2019 (an improvement on \cite{GuralnickTiep} from 2011) with strong bounds when the \(B\)--fixed points \(V^B = 0\) for \(B\) a Borel subgroup of \(G\). In defining characteristic \(r = p\), we instead have \cite{H1Defining} from 2013 stating that, for any irreducible \(kG\)-module \(V\),
\[\dim \H^1(G,V) \leq \max \left\{ \frac{z_p^{\frac{h^3}{6} + 1} - 1}{z_p - 1}, \frac{1}{2} (h^2(3h - 3)^3)^{\frac{h^2}{2}}\right\},\]
where \(z_p = \floor{\frac{h^3}{6} (1 + \log_p (h-1))}\) and \(h\) is the Coxeter number of \(G\). Referring to the above module for \(\E_6(q)\), this bound amounts to roughly \(10^{703}\) compared to the \(10^6\) that is the largest known example.

For second cohomology and higher, less is known but for example in \cite{GuralnickPresentations} from 2007 it is shown that \(\dim \H^2(G,V) \leq \frac{35}{2} \dim V\) for \(G\) any finite quasisimple group and \(V\) any \(G\)-module, or \(\frac{37}{2} \dim V\) for \(G\) any finite group and \(V\) a faithful, irreducible \(G\)-module.

In some cases, we also know that a bound of a particular type exists, but we have no explicit expressions for them. For a finite group of Lie type in its defining characteristic, Cline, Parshall and Scott \cite{ClineParshallScott} showed that \(\dim \H^1(G,V)\) (for \(V\) irreducible) is bounded by some constant dependent only upon the rank of \(G\). More generally, Guralnick and Tiep \cite{GuralnickTiep2} have showed that if \(G\) is a finite group and \(V\) an irreducible \(kG\)-module then \(\dim \H^1(G,V)\) is bounded by some constant dependent only on \(r\) and the \emph{sectional \(r\)-rank} of \(G\) (the maximum \(r\)-rank of an elementary abelian section of \(G\)). Very recently, the as-yet unpublished sequel to this work \cite{GuralnickTiep3} generalises the result to show that, for arbitrary \(D\) and irreducible \(V\), \(\dim \Ext_G^n(D, V)\) is bounded by a constant dependent only on the sectional \(r\)-rank of \(G\), \(\dim D\) and \(n\).

In this work, we generalise Theorem 2.2 from \cite{GuralnickTiep} to a larger class of groups, higher cohomology and reducible modules. Let \(O_{r'}(H)\) be the largest normal subgroup of a group \(H\) with order not divisible by \(r\) and let \(O^r(H)\) be the smallest normal subgroup of \(H\) such that \(H/O^r(H)\) is an \(r\)-group. Then we show the following.

\begin{introthm}
	Let \(G\) be a finite group and let \(k\) be an algebraically closed field of characteristic \(r\). Suppose \(H \leq G\) is such that \(O_{r'}(H) = O^r(H)\) and suppose that \(\perm\) is the permutation module \(\Ind_H^G k\) of \(G\) acting on cosets of \(H\). Suppose \(V\) is a \(kG\)-module with \(V^H = 0\). Then \(\H^n(G,V) \cong \Ext_G^{n-1}(V^*, \permG)\).
\end{introthm}

In fact, \(\perm^G \cong k\) and so in future we shall talk about \(\permk\) rather than \(\permG\). Of particular note is that when \(n = 1\) one may then determine \(\dim \H^1(G,V)\) by investigating the structure of \(\soc (\permk)\). More formally,

\begin{introcor}
	With the above notation, if \(V^H = 0\), then \(\H^1(G,V) \cong \GHom(V^*, \permk)\).
\end{introcor}

The hypothesis \(O_{r'}(H) = O^r(H)\) is vital to our arguments for the above theorem. It stems from \cref{NewGTLemma}, which one can see to be a generalisation of \cite[Lemma 2.1]{GuralnickTiep} by taking \(B\) to be a Borel subgroup of a group of Lie type \(G\) and \(r\) to be any prime not equal to the defining characteristic of \(G\). Using this lemma, specifically the statement that \(V^H = 0\) if and only if \(V^{O_{r'}(H)} = 0\), we can show that in fact \(\H^n(H,V) = 0\) for all \(n\) when \(V^H = 0\) and then a long exact sequence in \(\Ext\) implies the main theorem. The use of this hypothesis permits the extension of \cite[Theorem 2.2]{GuralnickTiep} from being a statement about irreducible modules for groups of Lie type to being a statement about any module for an arbitrary finite group.

If one wanted to obtain results similar to those of Guralnick--Tiep for all groups of Lie type, one approach may be to use induction on the rank of \(G\). In order to do this, we would need to know the cohomology of the rank 1 groups of Lie type. In any case, it is interesting to know what happens for the small rank groups of Lie type as this may be of use in forming conjectures or performing calculations in higher ranks. As such, the bulk of this article is dedicated to determining \(\Ext_G^i(V, W)\) for cross characteristic irreducible \(kG\)-modules \(V\), \(W\) for \(G \in \{\PSL_2(q), \PGL_2(q), \SL_2(q)\}\).

Our main tool for determining \(\Ext\) groups is as follows. Let \(\epsilon \colon \PC(V) \surj V\) be a (surjective) map onto \(V\) from its projective cover. We define the {\itshape Heller translate} \(\Omega V\) to be the kernel of this map, and note that by \cref{OmegaCohomology} for \(W\) irreducible, \(\dim \Ext_G^i(V, W) = \dim \Hom(\Omega^i V, W)\) (where \(\Omega^i V \coloneqq \Omega (\Omega^{i-1} V)\)) and so the dimension of \(\Ext_G^i(V, W)\) may be read off from the structure of \(\Omega^i V\). Then to determine the dimensions of these \(\Ext\) groups, we need only examine these Heller translates.

To examine Heller translates of particular modules, it is useful to know the structure of the projective indecomposable modules (PIMs) for \(G\). When examining a block with a cyclic defect group (a sufficient condition for this is that the Sylow \(r\)-subgroups of \(G\) are cyclic), the structure of each PIM for this block is easily described by the block's Brauer tree (see \cite[\nopp V]{AlperinLocal}). In such a case we also know that \(\dim \Ext_G^i(V, W) \leq 1\) for all indecomposable \(V\), irreducible \(W\) so we need only determine whether \(\Ext_G^i(V, W)\) is zero.

If the Sylow \(r\)-subgroups of \(G\) are not cyclic, the representation theory is much more complicated. When \(r > 2\), the {\itshape representation type} of any \(r\)-block of \(kG\) with noncyclic defect group is {\itshape wild}, loosely meaning that the representation theory of this block `contains' that of all finite-dimensional \(k\)-algebras. When \(r = 2\) and the Sylow 2-subgroups of \(G\) are dihedral, semidihedral or generalised quaternion then the representation type of any 2-blocks of \(G\) with noncyclic defect, while still not {\itshape finite} as in the cyclic case, is {\itshape tame}. As a result of this, the representation theory in such cases is slightly simpler. For \(G = \PSL_2(q)\), the Sylow 2-subgroups of \(G\) are dihedral and thus the representation type is tame. In this case, we have a description of the structure of the projective indecomposable \(kG\)-modules similar to the cyclic case via {\itshape Brauer graphs} as explained in \cite{PSL2BrauerGraphs}.

Using this description we are able to determine all \(\Ext\) groups for the two nontrivial modules in the principal block of \(kG\), completing the description of \(\Ext\) groups for this block. We then use the information obtained for \(\PSL_2(q)\) to extend these results to \(\SL_2(q)\) and \(\PGL_2(q)\) in all nondefining characteristics. As with \(\PSL_2(q)\), if \(r \neq 2\) this is almost immediate and the characteristic 2 case requires various additional considerations. For \(\PGL_2(q)\) we make use of the fact that \(\PGL_2(q) \leq \PSL_2(q^2)\) and use knowledge of \(\PSL_2(q)\) and Shapiro's Lemma to obtain the result and for \(\SL_2(q)\) we note that the Sylow 2-subgroups of \(G\) are generalised quaternion and thus all irreducible modules are periodic of period dividing 4, so \(\Omega^4 V \cong V\) for all irreducible \(V\).

\section{General results}

\begin{notation*}
	For the remainder of this article, let \(k\) be an algebraically closed field of characteristic \(r\), \(G\) be a finite group and, unless otherwise specified, \(\Hom\) shall denote \(\GHom\).
\end{notation*}

We begin by generalising a lemma and theorem from \cite{GuralnickTiep}, whilst also providing a shorter proof of the original theorem as a special case. Recall that \(O_{r'}(H)\) is the largest normal subgroup of a group \(H\) of order not divisible by \(r\) and \(O^r(H)\) is the smallest normal subgroup of \(H\) such that \(H/O^r(H)\) is an \(r\)-group.

\begin{lemma} \label{NewGTLemma}
	Suppose \(H\) is a finite group, \(k\) a field of characteristic \(r\) and \(V\) a \(kH\)-module. If \(O_{r'}(H) = O^r(H) \eqqcolon A\) then the following statements are equivalent.

	\begin{enumerate}[i)]
		\item \(V^H \neq 0\),
		\item \(H\) has trivial composition factors on \(V\),
		\item \(V^A \neq 0\),
		\item \((V^*)^H \neq 0\).
	\end{enumerate}
\end{lemma}

\begin{proof}
	That i) implies ii) is clear, and that ii) implies iii) follows from the fact that \(r \nmid \abs{A}\). To see that iii) implies i), we note that the \(r\)-group \(H/A\) acts on \(V^A\) so we have that \((V^A)^{H/A} \neq 0\) and so \(V^H \neq 0\). Thus i)--iii) are equivalent, it remains only to show that iv) is equivalent to the rest; to see this, note that using i) \(\iff\) iii) we have
	\[V^H \neq 0 \iff V^A \neq 0 \iff (V^*)^A \neq 0 \iff (V^*)^H \neq 0,\]
	as required.
\end{proof}

\begin{lemma} \label{HHomis0}
	Suppose \(H\) and \(A\) are as above. Then if \(V\) is a \(kH\)-module with \(V^H = 0\), we have that \(\H^n(H,V) = 0\) for all \(n\).
\end{lemma}

\begin{proof}
	Note that we have an exact sequence of groups \(1 \to A \to H \to H/A \to 1\), and thus from the Hochschild--Serre spectral sequence \cite[\nopp6.8.2]{Weibel} we can see that \(\H^n(H,V)\) is a subquotient of
	\[\bigoplus_{i + j = n} \H^i(H/A, \H^j(A,V)) \cong \H^n(H/A, V^A) \oplus \bigoplus_{\substack{i + j = n \\ i \neq n}}\H^i(H/A, \H^j(A,V)).\]
	We may then note that \(V^H = 0\) so \(V^A = 0\) by \cref{NewGTLemma} and \(\H^j(A,V) = 0\) for all \(j > 1\) since \(r \nmid \abs{A}\), thus the above module is 0 and so \(\H^n(H,V) = 0\) for all \(n\).
\end{proof}

\begin{thm} \label{GeneralTheorem}
	Suppose \(G\) is a finite group, \(k = \overline{k}\) a field of characteristic \(r\), \(V\) a \(kG\)-module and \(H \leq G\) with \(O_{r'}(H) = O^r(H) \eqqcolon A\). Let \(\perm\) be the permutation module of \(G\) acting on cosets of \(H\). Then if \(V^H = 0\), we have that \(\dim \H^n(G,V) = \dim \Ext_G^{n-1}(V^*, \perm/\perm^G) = \dim \Ext_G^{n-1}(V^*, \permk)\).
\end{thm}

\begin{proof}
	We first note that we have the exact sequence
	\[0 \to k \to \perm \to \permk \to 0\]
	and we may apply \(\Hom_G(V^*, -)\) to this to obtain a long exact sequence with segments of the form
	\[\cdots \to \Ext_G^{n-1}(V^*, \perm) \to \Ext_G^{n-1}(V^*, \permk) \to \Ext_G^n(V^*, k) \to \Ext_G^n(V^*, \perm) \to \cdots\]
	which, using Shapiro's Lemma \cite[Corollary 2.8.4]{Benson1}, reduces to
	\[\cdots \to \H^{n-1}(H, V) \to \Ext_G^{n-1}(V^*, \permk) \to \H^n(G,V) \to \H^n(H,V) \to \cdots. \tag{\(\star\)} \label{BestSequence}\]
	Then, applying \cref{HHomis0} it follows that \(\H^n(G,V) \cong \Ext_G^{n-1}(V^*, \permk)\) as required.
\end{proof}

An important observation is that when \(n = 1\) we get

\begin{cor} \label{H1Result}
	Retain all of the above notation. If \(V^H = 0\), then \(\dim \H^1(G,V) = \dim \Hom(V^*, \permk)\).	
\end{cor}

In particular, when \(V\) is irreducible we have that \(\dim \H^1(G,V)\) is the multiplicity of \(V^*\) in the socle of \(\permk\). This also gives the following corollary.

\begin{cor} \label{PermCompositionFactor}
	Retain notation as above. Let \(V\) be an irreducible \(kG\)-module which is not a composition factor of \(\perm\). Then \(\H^1(G,V) = 0\).
\end{cor}

Or, for a consequence of this corollary that can simply be read off from character degrees, we have the following.

\begin{cor}
	Let \(G\) be a finite group and \(V\) an irreducible \(kG\)-module. If \(\dim V > [G : H]\) for some \(H \leq G\) such that \(O_{r'}(H) = O^r(H)\), then \(\H^1(G,V) = 0\).
\end{cor}

One example where the previous corollaries are useful is for \(G = S_p\), the symmetric group of odd prime degree \(p\). \(G\) has a subgroup \(S_{p-1}\) of index \(p\), and \(O_{p'}(S_{p-1}) = O^p(S_{p-1}) = S_{p-1}\) as it has \(p'\)-order. Then by the previous corollary we can say that \(\H^1(G,V) = 0\) for any irreducible \(V\) of dimension greater than \(p\), or using \cref{PermCompositionFactor} we can say that \(\H^1(G,V) = 0\) for all nonlinear \(V\) apart from the unique nontrivial constituent of the permutation module on \(S_{p-1}\). Note that \(\H^1(G,k) = \Hom(G/G', k) = 0\) since \(p\) is odd and \([G : G'] = 2\) (where \(G' \coloneqq [G, G]\)).

A major application of these results is to the area of finite groups of Lie type in cross characteristic, since any Borel subgroup of such groups satisfies the conditions on \(H\) of all above theorems. By a finite group of Lie type here we mean a group \(\bm{G}^F\) as defined in \cite[31]{CarterFinite}, including twisted groups, or any quotients or subgroups of \(\bm{G}^F\) with the same socle. This then gives the following theorem as a special case.

\begin{thmcite}[{\cite[Theorem 2.2]{GuralnickTiep}}] \label{GTTheorem}
	Let \(G\) be a finite group of Lie type defined in characteristic \(p \neq r\) and let \(B \leq G\) be a Borel subgroup. For an irreducible \(kG\)-module \(V\) with \(V^B = 0\) we have that \(\dim \H^1(G,V)\) is the multiplicity of \(V^*\) in \(\soc \permk\). Alternatively, \(\dim \H^1(G,V) = \dim \Hom(V^*, \permk)\).
\end{thmcite}

Though our result also generalises to \(\H^n(G,V) \cong \Ext_G^{n-1}(V^*, \permk)\) and does not require that \(V\) be irreducible.

\begin{notation*}
For the remainder of the section, assume that \(G\) is a finite group and \(H \leq G\) such that \(O_{r'}(H) = O^r(H)\) and \(\perm \coloneqq \Ind_H^G k\).
\end{notation*}

We can use \cref{GeneralTheorem} repeatedly and the fact that \(\Ext_G^{n-1}(V^*, \permk) \cong \H^n(G, V \otimes \permk)\) to obtain \(\H^n(G,V)\) as the first cohomology of some tensor product of modules. 
Thus if enough can be understood about such tensor products, we could obtain another avenue by which to attack higher cohomology. Further, the following result is immediate from \eqref{BestSequence}.

\begin{propn} \label{NoDivideH}
	Let \(H \leq G\) with \(r \nmid \abs{H}\) and \(V\) a \(kG\)-module. Then
	\[\dim H^1(G,V) = \dim \Hom(V^*, \permk) - \dim V^H\]
	and
	\[\H^n(G,V) \cong \Ext_G^{n-1}(V^*, \permk) \cong \H^{n-1}(G, V \otimes \permk)\]
	for any \(n \geq 2\).
\end{propn}

An important note here is that this retrieves the main result for higher cohomology without requiring that \(V^H = 0\), and gives an easy calculation to determine \(\H^1(G,V)\) when \(V^H \neq 0\).

Often, it is far easier to determine \(\Hom\) between modules than it is to determine \(\Ext_G^n\). As such, we repeatedly apply \cref{NoDivideH} to obtain the following result which does not require explicit calculation of \(\Ext\).

\begin{propn} \label{HomIsEasierThanExt}
	Let \(H \leq G\) such that \(r \nmid \abs{H}\) and \(V\) be a \(kG\)-module. Then 
	\[\dim \H^2(G,V) = \dim \GHom(V^*, (\permk)^{\otimes 2}) - \dim \Hom_H(V^*, \permk).\]
\end{propn}

\begin{proof}
	Applying \cref{NoDivideH} to \(\H^2(G,V)\), then again to \(\H^1(G, V \otimes \permk)\) gives
	\begin{align*}
		\dim \H^2(G,V)					&= \dim \H^1(G,V \otimes \permk) 	\\
										&= \dim \GHom((V \otimes \permk)^*, \permk) - \dim (V \otimes \permk)^H \\
										&= \dim \GHom(V^*, (\permk)^{\otimes 2}) - \dim \Hom_H(k, V \otimes \permk) \\
										&= \dim \GHom(V^*, (\permk)^{\otimes 2}) - \dim \Hom_H(V^*, \permk)
	\end{align*}
	as claimed.
\end{proof}

As with \cref{GTTheorem}, the above result reduces the study of \(\H^2(G, V)\) for any \(V\) to the study of the structure of the fixed modules \((\permk)^{\otimes 2}\) and \(\permk\), and in the case where \(V^H = 0\) one may use Mackey's formula to obtain more information (which may be particularly useful if \(H = B\) is a Borel subgroup). A very similar method generalises this to all \(n\) with increasing tensor powers of \(\permk\).

To close out the section, we give some results and an important piece of notation which shall feature heavily in the remainder of this article.

\begin{defncite}[{\cite[46]{DavidGuidebookRepTheory}}]
	Let \(V\) be a \(kG\)-module with projective cover \(\PC(V)\). Then we define the \emph{Heller translate} or \emph{syzygy} \(\Omega V\) of \(V\) to be the kernel of the projective covering map \(\PC(V) \surj V\). So \(\Omega V\) is a module such that 
	\[0 \to \Omega V \to \PC(V) \to V \to 0\]
	is exact. We then define \(\Omega^{i+1} V \coloneqq \Omega (\Omega^i V))\). 
\end{defncite}

The Heller translate has a variety of nice properties. The most useful for our purposes will be the following two lemmas.

\begin{lemmacite}[{\cite[Proposition 1]{HellerIndecomposable}}] \label{HellerIndecomposable}
	The Heller translate \(\Omega\) is a permutation on the set of isomorphism classes of non-projective indecomposable \(kG\)-modules.
\end{lemmacite}

In particular, this tells us that the Heller translate of an indecomposable module is itself indecomposable. 

\begin{lemmacite}[{\cite[Lemma 1]{AlperinOmega}}] \label{OmegaCohomology}
	Let \(U\), \(V\) be \(kG\)-modules with \(V\) irreducible. Then, for any \(n \geq 0\),
	\[\Ext_G^n(U,V) \cong \Hom_G(\Omega^n U, V).\]
\end{lemmacite}

\begin{defn} \label{PeriodicDef}
	A module \(V\) is said to be \emph{periodic} of \emph{period} \(n\) if \(\Omega^n V \cong V\) for some \(n \geq 1\).
\end{defn}

From \cref{OmegaCohomology} we see that one can determine \(\dim \Ext_G^n(U,V)\) by looking at the multiplicity of the irreducible module \(V\) in the head of \(\Omega^n U\), so \(\H^n(G,V)\) is simply the multiplicity of \(V\) in \(\head \Omega^n k\). Similarly, by noting that \(\H^n(G,V) \cong \Ext_G^n(V^*, k) \cong \Hom_G(\Omega^n V^*, k)\), we can see that \(\H^n(G,V)\) is also the multiplicity of \(k\) in \(\head \Omega^n V^*\). Combining this with the definition above, we see that periodic modules have periodic cohomology.

\begin{propn} \label{LonelyModule}
	Let \(B\) be a block containing a single non-projective simple module \(V\) with a cyclic defect group. Then \(\Ext_G^n(V,V) \cong k\) for all \(n\).
\end{propn}

\begin{proof}
	Since \(V\) is not projective and is the only module in its block, we must have that \(V\) lies in the head of \(\Omega^n V\) for all \(n\). Thus \(\Ext_G^n(V,V) \neq 0\) for all \(n\). But since the defect group of \(B\) is cyclic, we know that \(\dim \Ext_G^n(V,V) \leq 1\) for all \(n\) and so we are done.
\end{proof}

\section{Cohomology of \texorpdfstring{\(\PSL_2(q)\)}{PSL(2,q)}} \label{sec:PSL2Cohomology}

Let \(p\) be a prime, \(G \coloneqq \PSL_2(q)\) for \(q = p^n\) and let \(k\) be an algebraically closed field of characteristic \(r \neq p\). In this section, we determine the cohomology of all irreducible \(kG\)-modules \(V\). We reproduce the character table for \(\PSL_2(q)\) for \(q \equiv 1 \mod 4\) below for convenience; the other cases are similar and may be found in \cite{BurkhardtPSLDecomposition}. There are three cases to consider, \(q \equiv 1 \mod 4\), \(q \equiv 3 \mod 4\) and \(q\) even. In \cref{1mod4CharacterTable}, \(q \equiv 1 \mod 4\), \(\gamma\) and \(\delta\) are distinct elements of order \(p\), and \(\alpha\) and \(\beta\) are elements of maximum order dividing \(q-1\) and \(q+1\), respectively. Also, \(\rho\) and \(\sigma\) are primitive \(\abs{\alpha}\)\upth{} and \(\abs{\beta}\)\upth{} roots of unity in \(\mathbb{C}\), respectively. Further, \(i \leq \frac{1}{4}(q-5)\), and \(j\), \(l\), \(m \leq \frac{1}{4}(q-1)\) are positive integers.

\begin{table}[h]
	\[
		\begin{array}{ccccccc} \toprule
						& 	1 					&	\gamma 							& 	\delta 							& 	\alpha^l 						& 	\beta^m							\\	\midrule
			1 			& 	1 					&	1 								& 	1 								& 	1 								& 	1								\\
			\xi_1 		& 	\frac{1}{2}(q+1) 	&	\frac{1}{2} (1 + \sqrt{q}) 		& 	\frac{1}{2} (1 - \sqrt{q}) 		& 	(-1)^l 							& 	0								\\
			\xi_2 		& 	\frac{1}{2}(q+1) 	&	\frac{1}{2} (1 - \sqrt{q}) 		& 	\frac{1}{2} (1 + \sqrt{q}) 		& 	(-1)^l 							& 	0								\\
			\theta_j 	& 	q-1 				&	-1 								& 	-1 								& 	0 								& 	- (\sigma^{jm} + \sigma^{-jm})	\\
			\varphi 	& 	q 					&	0 								& 	0 								& 	1 								& 	-1								\\
			\chi_i 		& 	q+1 				&	1 								& 	1 								& 	\rho^{il} + \rho^{-il} 			& 	0								\\	\bottomrule
		\end{array}
	\]
	\caption{Character table for \(\PSL_2(q)\), \(q \equiv 1 \mod 4\).} \label{1mod4CharacterTable}
\end{table}

The specific decomposition of the ordinary characters into Brauer characters depends upon whether \(r\) divides \(q-1\), \(q+1\) or both (\emph{i.e.} \(r = 2\)) and we treat each of these cases individually when needed. However, there are some pieces of information that remain similar or the same across all cross characteristic cases and so we outline these before we begin. \label{PSL2Irreducibles}

In particular, the list of character degrees for \(G\) does not change too much and we can give consistent notation for the irreducible \(kG\)-modules for the remainder of this section. To find the construction of the irreducible characters in characteristic zero, see \cite[\textsection 38]{LarryA} and combine with \cite[\nopp I--VI, VIII]{BurkhardtPSLDecomposition} for an elementary description of the irreducible characters in the modular case. A thorough (less-elementary) discussion of this for \(\SL_2(q)\) may also be found in \cite[Chapter 9]{SL2Representations} and then the irreducible representations of \(\PSL_2(q)\) are those which have \(Z(\SL_2(q))\) in their kernels.

Fix a Borel subgroup \(B = Q \rtimes T \leq G\) with unipotent radical \(Q \in \Syl_p G\) and maximal torus \(T\), cyclic of order \(\frac{1}{(q-1, 2)}(q-1)\). We have a family of \(kG\)-modules of dimension \(q+1\) obtained by inducing linear \(kB\)-modules to \(G\). Let \(\{k = X_1, X_2, \ldots, X_a\}\) be a set of representatives of \(N_G(T)\) conjugacy classes of 1-dimensional irreducible \(kT\)-modules, inflated to \(B\). The module \(\perm \coloneqq \Ind_B^G k\) contains at most two nontrivial irreducible constituents, though the precise structure of this module varies on a case by case basis. When \(r \neq 2\), we denote the unique nontrivial constituent of \(\perm\) by \(V\) and when \(r = 2\) we denote the two nontrivial irreducible constituents of \(\perm\) by \(V\) and \(W\).

Suppose \(i > 1\). Except for one case when \(q \equiv 1 \mod 4\), \(\Ind_B^G X_i \eqqcolon M_i\) is irreducible of dimension dividing \(q+1\). When \(q \equiv 1 \mod 4\), we let \(M_a\) denote the reducible module so obtained. Then \(\Ind_B^G X_a \eqqcolon M_a\) has two distinct irreducible constituents of dimension \(\frac{1}{2}(q+1)\) which we shall denote \(M_{a1}\) and \(M_{a2}\). In a similar vein, \(G\) has modules of dimensions dividing \(q-1\) which are obtained as constituents of nontrivial modules induced from a Singer cycle, \(S\). Let \(\{N_1, \ldots, N_b\}\) denote the modules obtained in this way up to \(N_G(S)\)-conjugacy (see \cite[\textsection38, Steps 6, 7]{LarryA}). Except for \(q \equiv 3 \mod 4\), all of these modules are irreducible of dimension \(q - 1\). When \(q \equiv 3 \mod 4\), we choose notation so that \(N_b\) is reducible. Then \(N_b\) has two distinct irreducible constituents of dimension \(\frac{1}{2}(q-1)\) which we shall denote \(N_{b1}\) and \(N_{b2}\).

The modules described above form a complete set of irreducible modules for \(G\) up to isomorphism.

\subsection*{Case 1: \(r\) odd} \label{PSL2Odd}

We first deal with the case where \(r\) is odd. In these cases, the Sylow \(r\)-subgroups of \(G\) are cyclic and so the structure of the projective modules in a block is determined by the {\itshape Brauer tree} of the block. The reader unfamiliar with Brauer trees should consult \cite[\textsection 17]{AlperinLocal} for more. In this section, we show the following.

\begin{thm} \label{PSL2MinusOneSummary}
	Let \(G \coloneqq \PSL_2(q)\) and \(k\) be an algebraically closed field of odd characteristic \(r \mid q-1\). Then the modules \(N_i\) (including \(N_{b1}\) and \(N_{b2}\) when \(q \equiv 3 \mod 4\)) are projective and \(\H^n(G, M_i) = 0\) for all \(n\). Further, 
	\[
		\H^n(G,k) \cong \begin{cases}
			0 & n \equiv 1, \ 2 \mod 4,\\
			k & n \equiv 0, \ 3 \mod 4,
		\end{cases} \qquad 
		\H^n(G,V) \cong \begin{cases}
			0 & n \equiv 0, \ 3 \mod 4,\\
			k & n \equiv 1, \ 2 \mod 4.
		\end{cases}
	\]
	with \(\Ext_G^n(V, V) \cong \H^n(G, k)\) and \(\Ext_G^n(V, k) \cong \H^n(G, V)\) for all \(n\). Finally, when \(q \equiv 1 \mod 4\), let \(Y\), \(Z \in \{M_{a1}, M_{a2}\}\) with \(Y \ncong Z\). Then
	\[
		\Ext_G^n(Y, Y) \cong \begin{cases}
			0	&	n \equiv 1, \ 2 \mod 4,\\
			k	&	n \equiv 0, \ 3 \mod 4,
		\end{cases} \qquad 
		\Ext_G^n(Y, Z) \cong \begin{cases}
			0	&	n \equiv 0, \ 3 \mod 4,\\
			k	& 	n \equiv 1, \ 2 \mod 4.
		\end{cases}
	\]
\end{thm}

\begin{thm} \label{PSL2PlusOneSummary}
	Let \(G \coloneqq \PSL_2(q)\) and \(k\) be an algebraically closed field of odd characteristic \(r \mid q+1\). Then the modules \(M_i\) (including \(M_{a1}\) and \(M_{a2}\) when \(q \equiv 1 \mod 4\)) are projective and \(\H^n(G, N_i) = 0\) for all \(n\). Further, when the \(r\)-part of \(q+1\) is not 3, \(\Ext_G^n(V, V) \cong k\) for all \(n\)
	\[
		\H^n(G,k) \cong \begin{cases}
			0 & n \equiv 1, \ 2 \mod 4,\\
			k & n \equiv 0, \ 3 \mod 4,
		\end{cases} \qquad 
		\H^n(G,V) \cong \Ext_G^n(V, k) \cong \begin{cases}
			0 & n \equiv 0, \ 3 \mod 4,\\
			k & n \equiv 1, \ 2 \mod 4.
		\end{cases}
	\]
	If the \(r\)-part of \(q+1\) is 3, then instead \(\Ext_G^n(V, V) \cong \H^n(G, k)\) for all \(n\). Finally, when \(q \equiv 3 \mod 4\), let \(Y\), \(Z \in \{N_{b1}, N_{b2}\}\) with \(Y \ncong Z\). Then
	\[
		\Ext_G^n(Y, Y) \cong \begin{cases}
			0	&	n \equiv 1, \ 2 \mod 4,\\
			k	&	n \equiv 0, \ 3 \mod 4,
		\end{cases} \qquad 
		\Ext_G^n(Y, Z) \cong \begin{cases}
			0	&	n \equiv 0, \ 3 \mod 4,\\
			k	& 	n \equiv 1, \ 2 \mod 4.
		\end{cases}
	\]
\end{thm}

In both of the above theorems, the projective irreducible modules can easily be determined as they are those which have dimensions divisible by the order of a Sylow \(r\)-subgroup (see \cite[Theorem 3.18]{NavarroCharactersBlocks}), and the cohomology of all modules outside the principal block is trivial since \(\Ext_G^n(Y, Z) = 0\) for all \(n\) and all irreducible \(Y\), \(Z\) which do not lie in the same block (see \cite[\textsection 13, Proposition 3]{AlperinLocal}.

For odd characteristic, the Brauer trees of these blocks are given in \cite{BurkhardtPSLDecomposition}. In the below discussion, \(B_0\) denotes the principal block and \(B_1\) denotes the block containing only \(M_{a1}\) and \(M_{a2}\) or \(N_{b1}\) and \(N_{b2}\), dependent on which pair exists (and is not projective) in the given situation.

The blocks for \(G\) containing two modules fall into two cases: case 1 holds for \(B_0\) whenever \(r\) is odd and \(r \mid q + 1\) and case 2 holds for \(B_0\) whenever \(r\) is odd and \(r \mid q - 1\), along with \(B_1\) whenever it exists. 

For some exceptional multiplicity \(m\), cases 1 and 2 respectively have Brauer trees
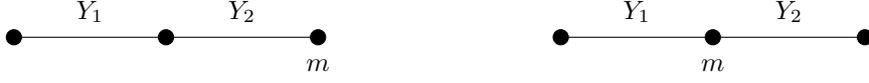
\[\begin{tikzpicture}
		\filldraw (-2,0) circle(0.1);
		\filldraw (0,0) circle(0.1);
		\filldraw (2,0) circle(0.1) node[below=5pt] {\(m\)};
		\draw (-1.9,0) -- (-0.1,0) node[pos=0.5,above=2pt] {\(Y_1\)};
		\draw (0.1,0) -- (1.9,0) node[pos=0.5,above=2pt] {\(Y_2\)};
	\end{tikzpicture} \qquad \qquad \qquad \qquad 
	\begin{tikzpicture}
		\filldraw (-2,0) circle(0.1);
		\filldraw (0,0) circle(0.1) node[below=5pt] {\(m\)};
		\filldraw (2,0) circle(0.1);
		\draw (-1.9,0) -- (-0.1,0) node[pos=0.5,above=2pt] {\(Y_1\)};
		\draw (0.1,0) -- (1.9,0) node[pos=0.5,above=2pt] {\(Y_2\)};
	\end{tikzpicture}\]
and Cartan matrices
\[\begin{pmatrix}
	2 & 1 \\ 1 & m+1
\end{pmatrix} \qquad \qquad \qquad \qquad \qquad \qquad
\begin{pmatrix}
	m+1 & m \\ m & m+1
\end{pmatrix}\]
where in case 1, the first column represents \(Y_1\) (corresponding to the trivial module for our purposes) and the second \(Y_2\) (corresponding to \(V\)). To describe the structure of the projective modules in this case, and in future, we use the following notation.

\begin{defn} \label{heartdef}
	Let \(M\) be a \(kG\)-module. The \emph{heart} \(\heart(M)\) of \(M\) is \((\rad M)/(\rad M \cap \soc M)\).
\end{defn}

\begin{defn} \label{ModuleShape}
	Given two \(kG\)-modules \(Y\) and \(Z\), we say \(Y \sim Z\) to indicate that \(Y\) and \(Z\) have the same radical factors, \emph{i.e.} \(\rad^{i-1} Y/\rad^i Y \cong \rad^{i-1} Z / \rad^i Z\) for all \(i\). Suppose that \(Y\) has radical factors \(Y_1\), \(Y_2\), \ldots, \(Y_n\). Then we denote the equivalence class of modules with this radical series by \([Y_1 \mid Y_2 \mid \ldots \mid Y_n]\) and say that \(Y \sim [Y_1 \mid Y_2 \mid \ldots \mid Y_n]\). Note that here we begin at the head and move through the radical layers, so that \(\head Y \cong Y_1\) and \(Y_n \inj \soc Y\). 

	Finally, given a uniserial \(kG\)-module \(Y\) with radical factors \(Y_1\), \ldots, \(Y_n\) and a semisimple \(kG\)-module \(Z\) we abuse this notation somewhat and write \([Y \mid Z]\) for a module of shape \([Y_1 \mid \ldots \mid Y_n \mid Z]\) and \([Y \oplus Z]\) for a module of shape \([Y_1 \oplus Z \mid \ldots \mid Y_n]\).
\end{defn}

Using the above trees, we see that in case 1, \(\PC(Y_1) \sim [Y_1 \mid Y_2 \mid Y_1]\) and \(\PC(Y_2)\) has heart \(Y_1 \oplus Z\) where \(Z \sim [Y_2 \mid Y_2 \mid \ldots \mid Y_2]\). In case 2, both projective covers are uniserial of the same length with \(\PC(Y_i) \sim [Y_i \mid Y_j \mid Y_i \mid \ldots \mid Y_j \mid Y_i]\) for \(i \neq j\). Note that when \(m = 1\) in case 1, we will get a different answer as \(\Ext_G^1(Y_2, Y_2) = 0\) and \(Z = 0\). This happens in \(\PSL_2(q)\) precisely when the \(r\)-part of \(q+1\) is 3.

\begin{propn} \label{PSL2OtherExtsCase1}
	Let \(B\) be a block with Brauer tree as in case 1 (for \(m \neq 1\)) with simple modules \(Y_1\), \(Y_2\). Then \(\Ext_G^n(Y_2,Y_2) \cong k\) for all \(n\), and, for \(i \neq j\),
	\[\Ext_G^n(Y_1,Y_1) \cong \begin{cases}
		0	&	n \equiv 1, \ 2 \mod 4,\\
		k	&	n \equiv 0, \ 3 \mod 4,
	\end{cases} \qquad 
	\Ext_G^n(Y_i,Y_j) \cong \begin{cases}
		0	&	n \equiv 0, \ 3 \mod 4,\\
		k	&	n \equiv 1, \ 2 \mod 4.
	\end{cases}\]
	If \(m = 1\), then instead \(\Ext_G^n(Y_2, Y_2) \cong \Ext_G^n(Y_1, Y_1)\) for all \(n\).
\end{propn}

\begin{proof}
	We proceed by computing \(\Omega^n Y_i\). This gives (immediately, since \(\Omega^n Y_1\) is uniserial for all \(n\)) \(\Omega^n Y_1\) of shapes \([Y_2 \mid Y_1]\), \([Z \mid Y_2]\) and \([Y_1 \mid Y_2]\) for \(n = 1\), 2, 3, and \(\Omega^4(Y_1) \cong Y_1\). When \(m = 1\), note that \(Z = 0\) and so \(\Omega^2 Y_1 \cong Y_2\) and so we can stop here (one may also deduce the required result from the fact that if \(m = 1\) then \(Y_1\) and \(Y_2\) may be swapped without loss of generality). Otherwise, when \(m > 1\), we must work harder.

	For \(Y_2\), the Heller translates are not all uniserial so some additional thought is required. We have \(\Omega Y_2 \sim [Y_1 \oplus Z \mid Y_2]\). The projective cover of this is then \(\PC(Y_1) \oplus \PC(Y_2)\), and the \(Y_2\) in \(\soc \Omega Y_2\) must stem from a diagonal submodule of \(Y_2 \oplus \soc Z\) (taking any module `higher up' in \(Z\) would result in \(\Omega^2 Y_2\) being decomposable) in \(\PC(Y_1) \oplus \PC(Y_2)\), yielding \(\Omega^2 Y_2 \sim [Y_2 \oplus Y_1 \mid Y_1 \oplus Y_2]\). Similarly, the \(Y_2\) in \(\soc \Omega^2 Y_2\) comes from a diagonal submodule of \(\head Z \oplus Y_2\) in \(\PC(Y_2) \oplus \PC(Y_1)\) and so we obtain \(\Omega^3 Y_2 \sim [Y_2 \mid Y_1 \oplus Z]\). From this it immediately follows that \(\Omega^4 Y_2 \cong Y_2\). The dimensions of \(\Ext_G^n\) may be read off from the shapes of these modules, and the result follows.
\end{proof}

\begin{propn} \label{PSL2OtherExtsCase2}
	Let \(B\) be a block with Brauer tree as in case 2 with simple modules \(Y_1\), \(Y_2\). Then, for \(i \neq j\), 
	\[\Ext_G^n(Y_i,Y_i) \cong \begin{cases}
		0	&	n \equiv 1, \ 2 \mod 4,\\
		k	&	n \equiv 0, \ 3 \mod 4,
	\end{cases} \qquad 
	\Ext_G^n(Y_i,Y_j) \cong \begin{cases}
		0	&	n \equiv 0, \ 3 \mod 4,\\
		k	& 	n \equiv 1, \ 2 \mod 4.
	\end{cases}\]
\end{propn}

\begin{proof}
	As before, we compute \(\Omega^n Y_i\). One sees immediately from the shape of the projective covers that \(\Omega Y_1 \sim \PC(Y_2)/Y_2\), and then \(\Omega^2 Y_1 \cong Y_2\). Since the choice of \(Y_1\) or \(Y_2\) was arbitrary, this determines the structure of all \(\Omega^n Y_i\) for both \(i\) and the result follows.
\end{proof}

The only remaining case is now made up of blocks containing a single non-projective simple module, \emph{i.e.} all modules not in \(B_0\) or \(B_1\) mentioned above (\(k\), \(V\), \(M_{a1}\), \(M_{a2}\), \(N_{b1}\) and \(N_{b2}\)) and not projective (so, not of dimension divisible by the order of a Sylow \(r\)-subgroup). This case is dealt with by \cref{LonelyModule}.

\begin{cor} \label{PSL2RemainingExts}
	Let \(Y\) be one of the following simple modules. 
	\begin{itemize}
		\item \(M_i\), \(i \nin \{1, a, a1, a2\}\), for odd \(r \mid q-1\)
		\item \(N_i\), \(i \nin \{b, b1, b2\}\), for odd \(r \mid q+1\)
		\item \(M_a\), for odd \(r \mid q-1\) and \(q \equiv 3 \mod 4\),
		\item \(N_b\), for odd \(r \mid q+1\) and \(q \equiv 1 \mod 4\).
	\end{itemize}
	Then \(Y\) is the only simple module in its block, is not projective, and \(\Ext_G^n(Y, Y) \cong k\) for all \(n\).
\end{cor}

Combining these results yields \cref{PSL2MinusOneSummary,PSL2PlusOneSummary}.

\subsection*{Case 2: \(r = 2\)} \label{PSL2Even}

Finally, we consider the case where \(\Char k = 2\). This is more difficult than the previous cases since the Sylow 2-subgroups of \(G\) are not cyclic.

The notation for this case is as set out on p.~\pageref{PSL2Irreducibles}. Here, \(\perm\) is indecomposable with radical factors \(k\), \(V \oplus W\), \(k\) and so again there are no nontrivial irreducible \(kG\)-modules with \(B\)--fixed points. When \(q \equiv 3 \mod 4\), \(\perm\) is projective (since \(B\) is a \(2'\)-group) and indecomposable with head \(k\) and is thus isomorphic to the projective cover \(\PC(k)\) of \(k\). The only irreducible modules lying in the principal block are \(k\), \(V\) and \(W\). The remaining irreducible modules do not behave too differently from the other cases, though the modules \(M_{a1}\), \(M_{a2}\), \(N_{b1}\) and \(N_{b2}\) do not exist in this case. Again, information on decomposition numbers may be found in Burkhardt's paper \cite[\nopp VIII]{BurkhardtPSLDecomposition}.

In this subsection, we prove the following two results.

\begin{thm} \label{PSLChar21mod4Summary}
	Let \(G \coloneqq \PSL_2(q)\) for \(q \equiv 1 \mod 4\) and suppose that \(k = \overline{\mathbb{F}_2}\). Then the modules \(N_i\) are projective for all \(i\), \(\H^n(G, M_i) = 0\) for all \(n\), \(i\) and
	\[\dim \H^n(G, k) = \begin{cases}
		\frac{n}{3} + 1 	&	n \equiv 0 \mod 3,\\
		\floor*{\frac{n}{3}}&	n \equiv 1 \mod 3,\\
		\ceil*{\frac{n}{3}}	&	n \equiv 2 \mod 3,
	\end{cases} \qquad
	\H^n(G,V) \cong \H^n(G,W) \cong \begin{cases}
		k & n \equiv 1 \mod 3, \\
		0 & \text{otherwise}.
	\end{cases}\]
	Further, \(\Ext_G^n(V, W) = \Ext_G^n(W, V) = 0\) for all \(n\) and \(\Ext_G^n(V, V) \cong \Ext_G^n(W, W)\) is zero for \(n \equiv 1 \mod 3\) and 1-dimensional otherwise. Finally, \(\Ext_G^n(M_i, M_i) \cong k\) for all \(n\).
\end{thm}

\begin{thm} \label{PSLChar23mod4Summary}
	Let \(G \coloneqq \PSL_2(q)\) for \(q \equiv 3 \mod 4\) and suppose that \(k = \overline{\mathbb{F}_2}\). Then the modules \(M_i\) are projective for all \(i\), \(\H^n(G, N_i) = 0\) for all \(n\), \(i\) and
	\[\dim \H^n(G, k) = \begin{cases}
		\frac{n}{3} + 1 	&	n \equiv 0 \mod 3,\\
		\floor*{\frac{n}{3}}&	n \equiv 1 \mod 3,\\
		\ceil*{\frac{n}{3}}	&	n \equiv 2 \mod 3,
	\end{cases} \qquad \qquad
	\dim \H^n(G,V) = \dim \H^n(G,W) = \ceil*{\frac{n}{3}}. \]
	Further, \(\dim \Ext_G^n(V, W) = \dim \Ext_G^n(W, V) = \ceil*{\frac{n}{3}}\) for all \(n\) and
	\[
		\dim \Ext_G^n(W,W) = \dim \Ext_G^n(V,V) = \begin{cases}
			\frac{n}{3} + 1 	&	n \equiv 0 \mod 3,\\
			\floor*{\frac{n}{3}}&	n \equiv 1 \mod 3,\\
			\ceil*{\frac{n}{3}}	&	n \equiv 2 \mod 3.
		\end{cases}
	\]
	Finally, \(\Ext_G^n(N_i, N_i) \cong k\) for all \(n\).
\end{thm}

We begin with the easy cases.

\begin{propn} \label{Char2StuffIsZero}
	If \(Y\) is one of the modules \(M_i\), \(N_j\) then \(\H^n(G,Y) = 0\) for all \(n\).
\end{propn}

\begin{proof}
	Referring to \cite[91]{BurkhardtPSLDecomposition} or \cite[Theorem 7.1.1]{SL2Representations}, we see that the above modules all lie outside the principal block. Thus \(\Ext_G^n(k, Y) = \H^n(G,Y) = 0\) for all \(n\) by \cite[\textsection 13, Proposition 3]{AlperinLocal}.
\end{proof}

\begin{propn} \label{Char2Projectives}
	When \(q \equiv 1 \mod 4\), \(N_i\) is projective for all \(i\). When \(q \equiv 3 \mod 4\), \(M_i\) is projective for all \(i\).
\end{propn}

\begin{proof}
	This follows from the fact that their dimensions are divisible by the order of a Sylow \(2\)-subgroup of \(G\).
\end{proof}

\begin{propn} \label{Char2LonelyModules}
	When \(q \equiv 1 \mod 4\), \(\Ext_G^n(M_i, M_i) \cong k\) for all \(n\) and when \(q \equiv 3 \mod 4\), \(\Ext_G^n(N_i, N_i) \cong k\) for all \(n\).
\end{propn}

\begin{proof}
	This follows from \cref{LonelyModule}. To see that we may apply this, note that the Sylow \(2\)-subgroups of \(G\) are dihedral and the principal block is the only 2-block of \(G\) of maximal defect and otherwise the defect group of any block \(B\) is a Sylow 2-subgroup of the centraliser of some 2-regular element by an equivalent definition of defect group \cite[278]{IsaacsCharacterTheory}. In particular, such a group must be cyclic if it is not the whole Sylow 2-subgroup of \(G\).
\end{proof}

We also know the cohomology of \(k\) due to Fong and Milgram:

\begin{propncite}[{\cite[Theorem 8.1]{PSL2TrivialCohomology}}]
	Let \(k\) be the trivial \(kG\)-module. Then \(\dim \H^n(G,k)\) is given by the coefficient of \(x^n\) in the Poincar\'{e} series
	\[P(x) = \frac{1+x^3}{(1-x^2)(1-x^3)} = 1 + x^2 + 2x^3 + x^4 + \ldots.\]
\end{propncite}

To see that the above does actually yield the statements from \cref{PSLChar21mod4Summary,PSLChar23mod4Summary}, note that using Maclaurin series,
\[(1+x^3) \frac{1}{1-x^2} \frac{1}{1-x^3} = (1+x^3) (1 + x^2 + x^4 + x^6 + \cdots) (1 + x^3 + x^6 + x^9 + \cdots)\]
and the coefficient \(a_n\) of \(x^n\) in such an expansion is the number of ways \(n\) may be written as \(2a + 3b\), for integers \(a\), \(b \geq 0\), plus the number of ways \(n-3\) may be written in the same form.
One may then check that this gives \(a_n = a_{n-3} + 1\) for \(n \geq 3\) with \(a_n = 0\) for \(n < 0\). It is therefore sufficient to check the first three terms of the series, as given above.

For the cohomology of the remaining two modules, however, we will need to bring in some extra machinery. We will need to make use of Brauer graphs and knowledge of the structure of the projective covers of modules for blocks with Brauer graphs.

An introduction to Brauer graphs may be found in \cite{PSL2BrauerGraphs} and \cite[\nopp 4.18]{Benson1}, so the unfamiliar reader should refer to these during what follows. We give below the Brauer graphs for \(\PSL_2(q)\) in characteristic 2, which differ dependent on whether \(q \equiv \pm 1 \mod 4\). These may be found in \cite[\textsection2, (4) \& (5)]{PSL2BrauerGraphs}, but the multiplicities of the vertices are not determined there. Thus before we give the Brauer graphs, we must determine the multiplicities at each vertex. We shall denote by \(m\) the multiplicity of the vertex at which both \(V\) and \(W\) are incident and determine the remaining multiplicities in the following lemmas. The case where \(q \equiv 3 \mod 4\) is easy:

\begin{lemma}
	Let \(q \equiv 3 \mod 4\). Then the multiplicities of the two vertices at which \(k\) is incident are both 1.
\end{lemma}

\begin{proof}
	This follows immediately from the fact that \(\perm\) is projective, since \(B\) is a \(2^{\prime}\)-group.
\end{proof}

\begin{lemma} \label{1mod4Multiplicities}
	Let \(q \equiv 1 \mod 4\). Then the multiplicities of the two vertices at which \(k\) is \emph{not} incident are both 1.
\end{lemma}

\begin{proof}
	From \cite[295]{ErdmannBlocks}, the Ext quiver of \(G\) in this case is below.
	\begin{center}
		\begin{tikzcd}[every arrow/.append style = {shift left}]
			V \arrow[r]		& \arrow[l] k \arrow[r] 	&	W \arrow[l]
		\end{tikzcd}
	\end{center}
	In the above quiver we see that no module is connected to itself. This means that \(\Ext_G^1(Y, Y) = 0\) for all irreducible \(Y\) in the principal block. In particular, \(\Ext_G^1(V, V) = \Ext_G^1(W, W) = 0\).
\end{proof}

Knowing the above facts, and the shape of the Brauer graphs for the principal block of \(\PSL_2(q)\) from \cite{PSL2BrauerGraphs}, we now give the Brauer graphs for \(G\) below. The multiplicity \(m\) at the labelled vertex is in each case a function of \(q\), however, note that our subsequent calculations only make use of the fact that \(m \neq 1\) and the conditions for this are given when required.

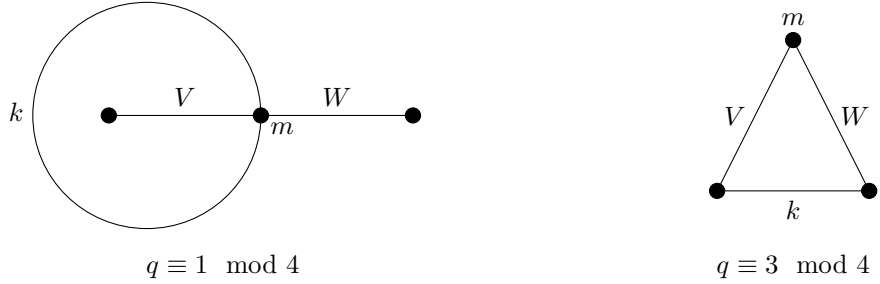
\begin{center}
	\begin{tikzpicture}

		\filldraw[color=black] (-7,0) circle(0.1);
		\filldraw[color=black] (-3,0) circle(0.1);
		\filldraw[color=black] (-5,0) circle(0.1) node[below=5pt,right=0pt] {\(m\)};

		\draw (-6.9,0) -- (-5.1,0) node[pos=0.5,above] {\(V\)};
		\draw (-4.9,0) -- (-3.1,0) node[pos=0.5,above] {\(W\)};
		\draw (-5,0) arc(0:356:1.5) node[pos=0.5,left] {\(k\)};

		\draw (-5.5,-2) node {\(q \equiv 1 \mod 4\)};

		\filldraw[color=black] (1,-1) circle(0.1);
		\filldraw[color=black] (3,-1) circle(0.1);
		\filldraw[color=black] (2,1) circle(0.1) node[above=2pt] {\(m\)};

		\draw (1,-1) -- node[pos=0.5,below] {\(k\)} (3,-1) -- node[pos=0.5,right] {\(W\)} (2,1) -- node[pos=0.5,left] {\(V\)} (1,-1);

		\draw (2,-2) node[align=center] {\(q \equiv 3 \mod 4\)};

	\end{tikzpicture}
\end{center}

Using this, we can see that the projective indecomposable modules for \(q \equiv 1 \mod 4\) are as below. Here we are using a bold letter to denote the first module in a new cycle (after the first) around a vertex. If a vertex has multiplicity \(m\), we take \(m\) anti-clockwise cycles around this vertex, starting from the edge labelled by the module whose projective cover we are interested in, and thus add \(m-1\) bold letters to our depiction of the corresponding projective cover.

\[ \label{1mod4Projectives}
	\begin{array}{ccc}
		V		\\
		k 		\\
		W 		\\
		k 		\\
		\bm{V}	\\
		\vdots 	\\
	 	W 		\\
		k 		\\
		V 				
	\end{array}
	\qquad \qquad \qquad
	\begin{array}{ccc}
		W			\\
		k 			\\
		V 			\\
		k 			\\
		\bm{W}		\\
		\vdots	 	\\
		V 			\\
		k 			\\
		W 			
	\end{array}
	\qquad \qquad \qquad
	\begin{array}{ccc}
				&	k 			& 			\\
		V 		& 				& 	W 		\\
		k 		& 				&	k 		\\
		W 		& 				&	V 		\\
		\bm{k}	& 	\bigoplus	& 	\bm{k}	\\
		\vdots	&				& 	\vdots 	\\
		k 		& 				& 	k 		\\
		W 		& 				& 	V 		\\
				&	k 			&			
	\end{array}
	\tag{\(\ddagger\)}
\]

and for \(q \equiv 3 \mod 4\) we get the following.

\[	\label{3mod4Projectives}
	\begin{array}{ccc}
			&	V			&			\\
			&				&	W 		\\
			&				&	\bm{V}	\\
			&				&	W		\\
		k	&	\bigoplus	&	\bm{V}	\\
			&				& 	\vdots 	\\
			& 				&	\bm{V}	\\
			&				&	W 		\\
			&	V 			&
	\end{array}
	\qquad \qquad \qquad
	\begin{array}{ccc}
				&	W			&			\\
		V		&				&	 		\\
		\bm{W}	&				&			\\
		V		&				&			\\
		\bm{W}	&	\bigoplus	&	k		\\
		\vdots	&				& 	 		\\
		\bm{W}	& 				&			\\
		V		&				&	 		\\
				&	W 			&
	\end{array}
	\qquad \qquad \qquad
	\begin{array}{ccc}
		&	k 	& 	\\
	V 	& 	\bigoplus 	& W \\
		& 	k 	&
	\end{array}
	\tag{\(\blacklozenge\)}
\]

We are now ready to determine the cohomology for \(G\).

\begin{propn} \label{1mod4}
	Suppose \(r = 2\) and \(q \equiv 1 \mod 4\). Then 
	\[\H^n(G,V) \cong \H^n(G,W) \cong \begin{cases}
		k & n \equiv 1 \mod 3, \\
		0 & \text{otherwise}.
	\end{cases}\]
\end{propn}

\begin{proof}
	This can be seen from the structure of \(\Omega^n V\) for \(n \leq 3\). To see this, note that we may find a projective resolution
	\[\cdots \xrightarrow{d_2} \PC(k) \xrightarrow{d_1} \PC(V) \xrightarrow{d_3} \PC(V) \xrightarrow{d_2} \PC(k) \xrightarrow{d_1} \PC(V) \surj V\]
	and analysis of this resolution shows that \(\head \Omega V \cong k\) and \(\head \Omega^2 V \cong \head \Omega^3 V \cong V\) and that this resolution (and thus \(V\)) is periodic of period 3. The analysis in this case is left as an exercise to the reader.
\end{proof}

For the next case, we again use the structure of \(\Omega^n V\) as above, but since the projective module we are investigating is not uniserial (and indeed none of \(k\), \(V\) or \(W\) are even periodic!) things are more complicated. We determine the structure of \(\Omega^n V\) by carefully following the diagonal extensions present at each step and determining which extensions must therefore be present in its Heller translate. We recommend that the reader draws the associated pictures for the below arguments (and those which follow) and keeps them in mind at all times as they should greatly aid understanding. We will present some such pictures in the below proof, but in future they will be omitted in the name of brevity.

\begin{propn} \label{3mod4}
	Suppose \(r = 2\) and \(q \equiv 3 \mod 4\) and let \(V\), \(W\) be as above. Then 
	\[\dim \H^n(G,V) = \dim \H^n(G,W) = \ceil*{\frac{n}{3}}. \]
\end{propn}

Recall that as our decision as to which module is \(V\) and which is \(W\) was arbitrary, it suffices to only prove the result for one. Before giving the proof, we must first set out some notation. Throughout this discussion and the proof which follows, the reader is encouraged to refer often to \eqref{3mod4Projectives}. Let \(Y_V\), \(Y_W\) be such that \(\heart \PC(V) \cong k \oplus Y_V\) and \(\heart \PC(W) \cong Y_W \oplus k\). Then we can easily see that \(\Omega V\) has shape \([k \oplus Y_V \mid V]\), which has projective cover \(\perm \oplus \PC(W)\).

Throughout the proof, we investigate extensions of a collection of six modules, appearing as submodules of \(\Omega V\) and \(\Omega k\). We briefly examine these below and in particular show that they are periodic.

We first observe that, since \(\head Y_V \cong W\), \(\PC(W) \surj [Y_V \mid V]\) with kernel of shape \([k \mid W]\), onto which \(\perm\) surjects with kernel of shape \([V \mid k]\) which is in turn covered by \(\PC(V)\) with kernel \([Y_V \mid V]\) again.

We also note that \(\PC(W) \surj [W \mid k]\) with kernel of shape \([Y_W \mid W]\), which is covered by \(\PC(V)\) with kernel \([k \mid V]\) and which one may in turn cover with \(\perm\) and kernel \([W \mid k]\), returning us to the module shape we started with. We illustrate the periodicity of the modules in question below.

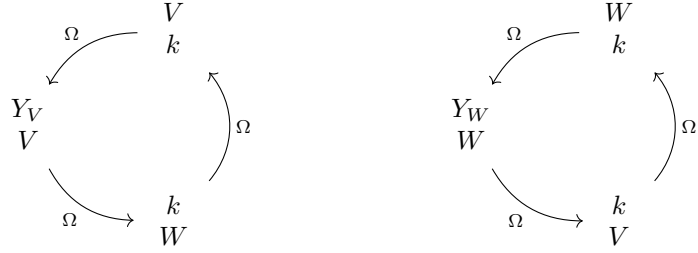
\begin{center}
	\begin{tikzcd}[row sep = 0.5em]
		\																&	\bigmoduleshape{V}{k} 	\arrow[bend right = 30,dl,"\Omega",swap]\\
		\bigmoduleshape{Y_V}{V} 	\arrow[bend right = 30,dr,"\Omega",swap]	&	\ \\
		\																&	\bigmoduleshape{k}{W}	\arrow[bend right = 40,uu,"\Omega",swap]	
	\end{tikzcd} \qquad \qquad \qquad
	\begin{tikzcd}[row sep = 0.5em]
		\															&	\bigmoduleshape{W}{k}	\arrow[bend right = 30, dl, "\Omega",swap] \\
		\bigmoduleshape{Y_W}{W} 	\arrow[bend right = 30,dr,"\Omega",swap]	&	\ \\
		\															&	\bigmoduleshape{k}{V}	\arrow[bend right = 40,uu,"\Omega",swap]
	\end{tikzcd}
\end{center}

We show that all of \(\Omega^n V\), aside from one irreducible constituent, is made up of a series of diagonal extensions of the above periodic modules. When \(n\) is odd \(\Omega^n V\) is a diagonal extension of \(k\) by two modules \(E_n\) and \(F_n\) and when \(n\) is even \(\Omega^n V\) is a diagonal extension of \(V\) by two modules which we shall also denote by \(E_n\) and \(F_n\). These modules are such that \(E_{2n}\) has shape 
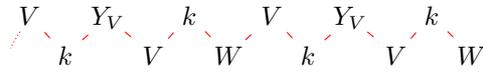
\begin{center}
	\begin{tikzpicture}
		\matrix(A)[matrix of math nodes, nodes in empty cells]{
			&	V	&		&	Y_V	&		&	k	&		&	V	&		&	Y_V	&		&	k	&		\\
			&		&	k	&		&	V	&		&	W	&		&	k	&		&	V	&		&	W	\\
		};
		\draw[red, densely dotted, shorten <>= 0.25cm] (A-1-2.center) -- (A-2-1.center);
		\draw[red, dashed, shorten <>= 0.3cm] (A-1-2.center) -- (A-2-3.center);
		\draw[red, dashed, shorten <>= 0.3cm] (A-2-3.center) -- (A-1-4.center);
		\draw[red, dashed, shorten <>= 0.3cm] (A-1-4.center) -- (A-2-5.center);
		\draw[red, dashed, shorten <>= 0.3cm] (A-2-5.center) -- (A-1-6.center);
		\draw[red, dashed, shorten <>= 0.3cm] (A-1-6.center) -- (A-2-7.center);
		\draw[red, dashed, shorten <>= 0.3cm] (A-2-7.center) -- (A-1-8.center);
		\draw[red, dashed, shorten <>= 0.3cm] (A-1-8.center) -- (A-2-9.center);
		\draw[red, dashed, shorten <>= 0.3cm] (A-2-9.center) -- (A-1-10.center);
		\draw[red, dashed, shorten <>= 0.3cm] (A-1-10.center) -- (A-2-11.center);
		\draw[red, dashed, shorten <>= 0.3cm] (A-2-11.center) -- (A-1-12.center);
		\draw[red, dashed, shorten <>= 0.3cm] (A-1-12.center) -- (A-2-13.center);
	\end{tikzpicture}
\end{center}
and \(F_{2n}\) has shape 
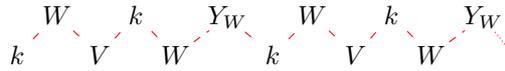
\begin{center}
	\begin{tikzpicture}
		\matrix(A)[matrix of math nodes, nodes in empty cells]{
			&	W	&		&	k	&		&	Y_W	&		&	W	&		&	k	&		&	Y_W	&		\\
		k	&		&	V	&		&	W	&		&	k	&		&	V	&		&	W	&		&		\\
		};
		\draw[red, densely dotted, shorten <>= 0.25cm] (A-1-12.center) -- (A-2-13.center);
		\draw[red, dashed, shorten <>= 0.3cm] (A-2-1.center) -- (A-1-2.center);
		\draw[red, dashed, shorten <>= 0.3cm] (A-1-2.center) -- (A-2-3.center);
		\draw[red, dashed, shorten <>= 0.3cm] (A-2-3.center) -- (A-1-4.center);
		\draw[red, dashed, shorten <>= 0.3cm] (A-1-4.center) -- (A-2-5.center);
		\draw[red, dashed, shorten <>= 0.3cm] (A-2-5.center) -- (A-1-6.center);
		\draw[red, dashed, shorten <>= 0.3cm] (A-1-6.center) -- (A-2-7.center);
		\draw[red, dashed, shorten <>= 0.3cm] (A-2-7.center) -- (A-1-8.center);
		\draw[red, dashed, shorten <>= 0.3cm] (A-1-8.center) -- (A-2-9.center);
		\draw[red, dashed, shorten <>= 0.3cm] (A-2-9.center) -- (A-1-10.center);
		\draw[red, dashed, shorten <>= 0.3cm] (A-1-10.center) -- (A-2-11.center);
		\draw[red, dashed, shorten <>= 0.3cm] (A-2-11.center) -- (A-1-12.center);
	\end{tikzpicture}
\end{center}
where each red dashed line indicates a non-split extension and the dotted line indicates that the pattern continues in this way. Each \(E_{2n}\) and \(F_{2n}\) contains \(n\) of the periodic modules mentioned above, so one can see that above we have drawn \(E_{12}\) and \(F_{12}\) (with possible continuations). Next, \(E_{2n+1}\) and \(F_{2n+1}\) have respective shapes
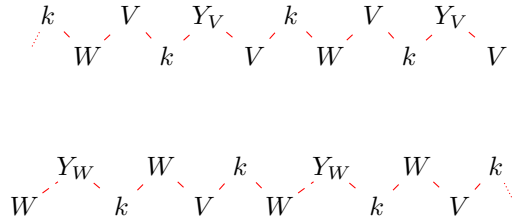
\begin{center}
	\begin{tikzpicture}
		\matrix(E)[matrix of math nodes, nodes in empty cells] at (0,0) {
			&	k	&		&	V	&		&	Y_V	&		&	k	&		&	V	&		&	Y_V	&		\\
			&		&	W	&		&	k	&		&	V	&		&	W	&		&	k	&		&	V	\\
		};
		\draw[red, densely dotted, shorten <>= 0.25cm] (E-1-2.center) -- (E-2-1.center);
		\draw[red, dashed, shorten <>= 0.3cm] (E-1-2.center) -- (E-2-3.center);
		\draw[red, dashed, shorten <>= 0.3cm] (E-2-3.center) -- (E-1-4.center);
		\draw[red, dashed, shorten <>= 0.3cm] (E-1-4.center) -- (E-2-5.center);
		\draw[red, dashed, shorten <>= 0.3cm] (E-2-5.center) -- (E-1-6.center);
		\draw[red, dashed, shorten <>= 0.3cm] (E-1-6.center) -- (E-2-7.center);
		\draw[red, dashed, shorten <>= 0.3cm] (E-2-7.center) -- (E-1-8.center);
		\draw[red, dashed, shorten <>= 0.3cm] (E-1-8.center) -- (E-2-9.center);
		\draw[red, dashed, shorten <>= 0.3cm] (E-2-9.center) -- (E-1-10.center);
		\draw[red, dashed, shorten <>= 0.3cm] (E-1-10.center) -- (E-2-11.center);
		\draw[red, dashed, shorten <>= 0.3cm] (E-2-11.center) -- (E-1-12.center);
		\draw[red, dashed, shorten <>= 0.3cm] (E-1-12.center) -- (E-2-13.center);

		\matrix(F)[matrix of math nodes, nodes in empty cells] at (0, -2) {
			&	Y_W	&		&	W	&		&	k	&		&	Y_W	&		&	W	&		&	k	&		\\
		W	&		&	k	&		&	V	&		&	W	&		&	k	&		&	V	&		&		\\
		};
		\draw[red, densely dotted, shorten <>= 0.25cm] (F-1-12.center) -- (F-2-13.center);
		\draw[red, dashed, shorten <>= 0.3cm] (F-2-1.center) -- (F-1-2.center);
		\draw[red, dashed, shorten <>= 0.3cm] (F-1-2.center) -- (F-2-3.center);
		\draw[red, dashed, shorten <>= 0.3cm] (F-2-3.center) -- (F-1-4.center);
		\draw[red, dashed, shorten <>= 0.3cm] (F-1-4.center) -- (F-2-5.center);
		\draw[red, dashed, shorten <>= 0.3cm] (F-2-5.center) -- (F-1-6.center);
		\draw[red, dashed, shorten <>= 0.3cm] (F-1-6.center) -- (F-2-7.center);
		\draw[red, dashed, shorten <>= 0.3cm] (F-2-7.center) -- (F-1-8.center);
		\draw[red, dashed, shorten <>= 0.3cm] (F-1-8.center) -- (F-2-9.center);
		\draw[red, dashed, shorten <>= 0.3cm] (F-2-9.center) -- (F-1-10.center);
		\draw[red, dashed, shorten <>= 0.3cm] (F-1-10.center) -- (F-2-11.center);
		\draw[red, dashed, shorten <>= 0.3cm] (F-2-11.center) -- (F-1-12.center);
	\end{tikzpicture}
\end{center}
where \(E_{2n+1}\) contains \(n+1\) of the periodic modules seen above and \(F_{2n+1}\) contains \(n\) (so we have drawn \(E_{11}\) and \(F_{13}\)). We will show that the shape of \(\Omega^{2n} V\) is as on the left below, and the shape of \(\Omega^{2n+1} V\) is as on the right. Here the \(k\) is part of a diagonal non-split extension with the \(V\) on the right hand side of \(E_{2n+1}\) and the \(W\) on the left hand side of \(F_{2n}\) and similarly the \(V\) is part of a diagonal non-split extension with \(W\) and \(k\).
\begin{center}
	\begin{tikzpicture}
		\matrix(left)[matrix of math nodes, nodes in empty cells] at (-2,0) {
				&	V	&			\\
		E_{2n}	&		&	F_{2n}	\\
		};

		\draw[red, dashed, shorten <>= 0.3cm] (left-2-1.center) -- (left-1-2.center);
		\draw[red, dashed, shorten <>= 0.3cm] (left-1-2.center) -- (left-2-3.center);

		\matrix(right)[matrix of math nodes, nodes in empty cells] at (2,0) {
				&	k	&				\\
		E_{2n+1}&		&	F_{2n+1}	\\
		};

		\draw[red, dashed, shorten <>= 0.3cm] (right-2-1.center) -- (right-1-2.center);
		\draw[red, dashed, shorten <>= 0.3cm] (right-1-2.center) -- (right-2-3.center);
	\end{tikzpicture}
\end{center}

In full, we claim the following.

\begin{lemma} \label{3mod4structure}
 	
	The module \(\Omega^{2n} V\) is of shape

	\begin{center}
		\begin{tikzpicture}
			\matrix(A)[matrix of math nodes, nodes in empty cells] {
					&	k	&		&	V	&		&	Y_V	&		&	k	&		&	\bm{V}	&		&	W	&		&	k	&		&	Y_W	&		&	W	&		\\
					&		&	W	&		&	k	&		&	V	&		&	W	&		&	k	&		&	V	&		&	W	&		&	k	&		&		\\
			};

			\draw[red, densely dotted, shorten <>= 0.25cm] (A-1-2.center) -- (A-2-1.center);
			\draw[red, densely dotted, shorten <>= 0.25cm] (A-1-18.center) -- (A-2-19.center);

			\draw[red, dashed, shorten <>= 0.3cm] (A-1-2.center) -- (A-2-3.center);
			\draw[red, dashed, shorten <>= 0.3cm] (A-2-3.center) -- (A-1-4.center);
			\draw[red, dashed, shorten <>= 0.3cm] (A-1-4.center) -- (A-2-5.center);
			\draw[red, dashed, shorten <>= 0.3cm] (A-2-5.center) -- (A-1-6.center);
			\draw[red, dashed, shorten <>= 0.3cm] (A-1-6.center) -- (A-2-7.center);
			\draw[red, dashed, shorten <>= 0.3cm] (A-2-7.center) -- (A-1-8.center);
			\draw[red, dashed, shorten <>= 0.3cm] (A-1-8.center) -- (A-2-9.center);
			\draw[red, dashed, shorten <>= 0.3cm] (A-2-9.center) -- (A-1-10.center);
			\draw[red, dashed, shorten <>= 0.3cm] (A-1-10.center) -- (A-2-11.center);
			\draw[red, dashed, shorten <>= 0.3cm] (A-2-11.center) -- (A-1-12.center);
			\draw[red, dashed, shorten <>= 0.3cm] (A-1-12.center) -- (A-2-13.center);
			\draw[red, dashed, shorten <>= 0.3cm] (A-2-13.center) -- (A-1-14.center);
			\draw[red, dashed, shorten <>= 0.3cm] (A-1-14.center) -- (A-2-15.center);
			\draw[red, dashed, shorten <>= 0.3cm] (A-2-15.center) -- (A-1-16.center);
			\draw[red, dashed, shorten <>= 0.3cm] (A-1-16.center) -- (A-2-17.center);
			\draw[red, dashed, shorten <>= 0.3cm] (A-2-17.center) -- (A-1-18.center);
		\end{tikzpicture}
	\end{center}

	and \(\Omega^{2n+1} V\) is of shape

	\begin{center}
		\begin{tikzpicture}
			\matrix(A)[matrix of math nodes, nodes in empty cells] {
					&	Y_V	&		&	k	&		&	V	&		&	Y_V	&		&	\bm{k}	&		&	Y_W	&		&	W	&		&	k	&		&	Y_W	&		\\
					&		&	V	&		&	W	&		&	k	&		&	V	&		&	W	&		&	k	&		&	V	&		&	W	&		&	\\
			};

			\draw[red, densely dotted, shorten <>= 0.25cm] (A-1-2.center) -- (A-2-1.center);
			\draw[red, densely dotted, shorten <>= 0.25cm] (A-1-18.center) -- (A-2-19.center);

			\draw[red, dashed, shorten <>= 0.3cm] (A-1-2.center) -- (A-2-3.center);
			\draw[red, dashed, shorten <>= 0.3cm] (A-2-3.center) -- (A-1-4.center);
			\draw[red, dashed, shorten <>= 0.3cm] (A-1-4.center) -- (A-2-5.center);
			\draw[red, dashed, shorten <>= 0.3cm] (A-2-5.center) -- (A-1-6.center);
			\draw[red, dashed, shorten <>= 0.3cm] (A-1-6.center) -- (A-2-7.center);
			\draw[red, dashed, shorten <>= 0.3cm] (A-2-7.center) -- (A-1-8.center);
			\draw[red, dashed, shorten <>= 0.3cm] (A-1-8.center) -- (A-2-9.center);
			\draw[red, dashed, shorten <>= 0.3cm] (A-2-9.center) -- (A-1-10.center);
			\draw[red, dashed, shorten <>= 0.3cm] (A-1-10.center) -- (A-2-11.center);
			\draw[red, dashed, shorten <>= 0.3cm] (A-2-11.center) -- (A-1-12.center);
			\draw[red, dashed, shorten <>= 0.3cm] (A-1-12.center) -- (A-2-13.center);
			\draw[red, dashed, shorten <>= 0.3cm] (A-2-13.center) -- (A-1-14.center);
			\draw[red, dashed, shorten <>= 0.3cm] (A-1-14.center) -- (A-2-15.center);
			\draw[red, dashed, shorten <>= 0.3cm] (A-2-15.center) -- (A-1-16.center);
			\draw[red, dashed, shorten <>= 0.3cm] (A-1-16.center) -- (A-2-17.center);
			\draw[red, dashed, shorten <>= 0.3cm] (A-2-17.center) -- (A-1-18.center);
		\end{tikzpicture}
	\end{center}

	where we have put the irreducible module between \(E_n\) and \(F_n\) in bold.
\end{lemma}

\begin{proof}
	We proceed by induction. For our base case, we observe that \(\Omega^0 V \cong V\) is already of the required form. The bold \(V\) in \(\Omega^{2n} V\) contributes an additional module of shape \([Y_V \mid V]\) to \(\Omega^{2n+1} V\) which we associate to \(E_{2n+1}\), and similarly the bold \(k\) in \(\Omega^{2n+1} V\) contributes an additional \([W \mid k]\) to \(\Omega^{2n+2} V\) which we associate to \(F_{2n+2}\).

	We first show that \(\Omega F_{2n} \sim F_{2n+1}\), \(\Omega E_{2n+1} \sim E_{2n+2}\), \(E_{2n+1}\) is an extension of \(\Omega E_{2n}\) by \([Y_V \mid V]\) and \(F_{2n+2}\) is an extension of \(\Omega F_{2n+1}\) by \([W \mid k]\). 

	We start with \(\Omega F_{2n} \sim F_{2n+1}\).	Due to the layout of the extensions as drawn above, it is sufficient to only consider a small case. To see this, note that \(F_6\) is a quotient of \(F_{2n}\) (by a submodule of shape \(F_{2n-6}\)) for all \(n \geq 3\). Noting that \(\PC(F_{2n}) \cong \PC(F_{2n-6}) \oplus \PC(F_6)\), one can easily show that \(\Omega F_{2n}\) is then a known extension of \(\Omega F_6\) by \(\Omega F_{2n-6}\).

	The smallest \(n\) from which we may obtain all the required information is \(6\), thus this is the case we shall consider. We first give \(F_6\) for convenience.
	\begin{center}
		\begin{tikzpicture}
			\matrix(F)[matrix of math nodes, nodes in empty cells]{
					&	W	&		&	k	&		&	Y_W	&	\\
				k	&		&	V	&		&	W	&		&	\\
			};

			\draw[red, dashed, shorten <>= 0.3cm] (F-2-1.center) -- (F-1-2.center);
			\draw[red, dashed, shorten <>= 0.3cm] (F-1-2.center) -- (F-2-3.center);
			\draw[red, dashed, shorten <>= 0.3cm] (F-2-3.center) -- (F-1-4.center);
			\draw[red, dashed, shorten <>= 0.3cm] (F-1-4.center) -- (F-2-5.center);
			\draw[red, dashed, shorten <>= 0.3cm] (F-2-5.center) -- (F-1-6.center);
		\end{tikzpicture}
	\end{center}
	\label{FirstUseOfArgument} The following argument will be vital for much of the rest of this section. In order to save trees, we will not be drawing all of the below representations of modules and their extensions in future. The projective cover of \(F_6\) is 
	\[\threemoduleshape{\redhigh{W}}{\bluehigh{Y_W} \oplus \redhigh{k}}{W} \bigoplus \threemoduleshape{\redhigh{k}}{\bluehigh{V \oplus W}}{k} \bigoplus \threemoduleshape{\redhigh{V}}{{k} \oplus \bluehigh{Y_V}}{V}\]
	where the quotients highlighted in red indicate areas in \(F_6\) in which there is no choice as to which projective module will cover it. Conversely, we indicate constituents in blue where the minimal submodule of the module containing these constituents is not determined up to isomorphism by its shape. In other words, the projective cover of the module (taken to refer to the surjective maps as well as the projective module) is not determined by the shape. For example, a module of shape \([V \oplus W \mid k]\) could (since such modules are not unique) have a covering by \(\PC(V) \oplus \PC(W)\) such that \(\PC(V) \surj [V \mid k]\) and \(\PC(W) \surj W\), one such that \(\PC(V) \surj V\) and \(\PC(W) \surj [W \mid k]\) or some covering in which the copy of \(k\) is diagonal between \(\PC(V)\) and \(\PC(W)\). 

	Taking note of the extensions in the second row of \(F_6\) (as drawn above), we see that \(V\) must be in an extension of both \(W\) and \(k\), and \(W\) in an extension of \(k\) and \(Y_W\). This indicates that these three modules must correspond to some diagonal quotient of the corresponding projective cover. Thus in the modules highlighted in blue above, our quotient involves some diagonal submodule of the constituents of adjacent (as written) projective covers. For example, the non-split extension of shape \([W \oplus k \mid V]\) must stem from taking \(V\) as a diagonal submodule of the subquotient \(\head Y_W \oplus V\) in \(\PC(W) \oplus \PC(k)\). Similarly, the non-split extension \([k \oplus Y_W \mid W]\) must stem from taking \(W\) as a diagonal submodule of the subquotient \(W \oplus \soc Y_V\) in \(\PC(k) \oplus \PC(V)\).

	It then remains to see what a covering map of the described shape has as its kernel. The diagonal submodules occurring in the previous paragraph give rise to diagonal extensions of the heads of the kernel of the covering map. For example, covering \([W \oplus k \mid V]\) as described above has kernel of shape \([k \oplus W \oplus V \mid Y_V \oplus k]\). Similarly, the described covering of \([k \oplus Y_W \mid W]\) has kernel \([V \oplus k \oplus W \mid k \oplus V]\). Patching these together, removing one copy of \(k\) from the head to deal with the \(k\) indicated in red above and being careful not to count any constituents of \(\PC(k)\) twice, we obtain \(\Omega F_6\) of shape
	\begin{center}
		\begin{tikzpicture}
			\matrix(F)[matrix of math nodes, nodes in empty cells]{
					&	Y_W	&		&	W	&		&	k	&	\\
				W	&		&	k	&		&	V	&		&	\\
			};

			\draw[red, dashed, shorten <>= 0.3cm] (F-2-1.center) -- (F-1-2.center);
			\draw[red, dashed, shorten <>= 0.3cm] (F-1-2.center) -- (F-2-3.center);
			\draw[red, dashed, shorten <>= 0.3cm] (F-2-3.center) -- (F-1-4.center);
			\draw[red, dashed, shorten <>= 0.3cm] (F-1-4.center) -- (F-2-5.center);
			\draw[red, dashed, shorten <>= 0.3cm] (F-2-5.center) -- (F-1-6.center);
		\end{tikzpicture}
	\end{center}
	with \(k\) joined to the \(V\) in the head of \(Y_W\). This is the same shape as \(F_7\), as required. A subset of this argument also yields \(\Omega F_2\) and \(\Omega F_4\).

	We next show that \(E_{2n+1}\) is an extension of \(\Omega E_{2n}\) by \([Y_V \mid V]\). As before, we need only consider a small case since \(E_{2n}\) is an extension of \(E_6\) by \(E_{2n-6}\) for all \(n \geq 3\) and so we give \(E_6\) below, where \(k\) is attached to the head of \(Y_V\).
	\begin{center}
		\begin{tikzpicture}
			\matrix(F)[matrix of math nodes, nodes in empty cells]{
				V	&		&	Y_V	&		&	k	&		\\
					&	k	&		&	V	&		&	W	\\
			};

			\draw[red, dashed, shorten <>= 0.3cm] (F-1-1.center) -- (F-2-2.center);
			\draw[red, dashed, shorten <>= 0.3cm] (F-2-2.center) -- (F-1-3.center);
			\draw[red, dashed, shorten <>= 0.3cm] (F-1-3.center) -- (F-2-4.center);
			\draw[red, dashed, shorten <>= 0.3cm] (F-2-4.center) -- (F-1-5.center);
			\draw[red, dashed, shorten <>= 0.3cm] (F-1-5.center) -- (F-2-6.center);
		\end{tikzpicture}
	\end{center}
	The projective cover of this, with colour highlights as in the previous case, is then
	\[\threemoduleshape{\redhigh{V}}{Y_V \oplus \bluehigh{k}}{V} \bigoplus \threemoduleshape{\redhigh{W}}{\bluehigh{k} \oplus \bluehigh{Y_W}}{W} \bigoplus \threemoduleshape{\redhigh{k}}{\bluehigh{V} \oplus \redhigh{W}}{k}.\]
	We may then proceed in a similar manner to the \(F_6\) case above. The extension of \(V \oplus Y_V\) by \(k\) must come from a diagonal \(k\) in the subquotient \(k \oplus k\) of \(\PC(V) \oplus \PC(W)\) and the extension \([Y_V \oplus k \mid V]\) can only come from a diagonal \(V\) in the subquotient \(\soc Y_W \oplus V\) of \(\PC(W) \oplus \PC(k)\). Patching together the extensions this gives us in \(\Omega E_6\), we obtain \(\Omega E_6\) of shape
	\begin{center}
		\begin{tikzpicture}
			\matrix(F)[matrix of math nodes, nodes in empty cells]{
				Y_V	&		&	k	&		&	V	&		\\
					&	V	&		&	W	&		&	k	\\
			};

			\draw[red, dashed, shorten <>= 0.3cm] (F-1-1.center) -- (F-2-2.center);
			\draw[red, dashed, shorten <>= 0.3cm] (F-2-2.center) -- (F-1-3.center);
			\draw[red, dashed, shorten <>= 0.3cm] (F-1-3.center) -- (F-2-4.center);
			\draw[red, dashed, shorten <>= 0.3cm] (F-2-4.center) -- (F-1-5.center);
			\draw[red, dashed, shorten <>= 0.3cm] (F-1-5.center) -- (F-2-6.center);
		\end{tikzpicture}
	\end{center}
	We then note that `attaching' \([Y_V \mid V]\) on the right via an extension \([Y_V \mid k \oplus V]\) gives us the shape of \(E_7\), as required. A subset of this argument also yields \(\Omega E_2\) and \(\Omega E_4\).

	Applying exactly the same argument (using the same extensions seen above and noting that \(E_5\) and \(F_7\) are respective quotients of \(E_{2n+1}\) and \(F_{2n+1}\) by \(E_{2n - 5}\) and \(F_{2n - 5}\)) yields \(\Omega E_{2n+1} \sim E_{2n+2}\) and we also get \(F_{2n+2}\) as an extension of \(\Omega F_{2n+1}\) by \([W \mid k]\) via an extension \([W \mid k \oplus V]\).

	We now return to \(\Omega^n V\). In particular, we need to deal with the interface between \(E_n\), \(F_n\) and \(V\) or \(k\). Suppose that we have a module of the required form for some even \(n\). Then the \(V\) in the middle is part of an extension of shape \([V \mid W \oplus k]\). To cover this extension, we would need to take a diagonal submodule of the subquotient \(W \oplus \head Y_V\) of \(\PC(k) \oplus \PC(V)\) and a diagonal \(k\) in the subquotient \(k \oplus k\) of \(\PC(V) \oplus \PC(W)\). Such a covering would then have kernel of shape
	\begin{center}
		\begin{tikzpicture}
			\matrix(F)[matrix of math nodes, nodes in empty cells]{
				V	&		&	Y_V	&		&	\bm{k}	&		&	Y_W	\\
					&	k	&		&	V	&			&	W	&		\\
			};

			\draw[red, dashed, shorten <>= 0.3cm] (F-1-1.center) -- (F-2-2.center);
			\draw[red, dashed, shorten <>= 0.3cm] (F-2-2.center) -- (F-1-3.center);
			\draw[red, dashed, shorten <>= 0.3cm] (F-1-3.center) -- (F-2-4.center);
			\draw[red, dashed, shorten <>= 0.3cm] (F-2-4.center) -- (F-1-5.center);
			\draw[red, dashed, shorten <>= 0.3cm] (F-1-5.center) -- (F-2-6.center);
			\draw[red, dashed, shorten <>= 0.3cm] (F-2-6.center) -- (F-1-7.center);
		\end{tikzpicture}
	\end{center}
	with \(k\) attached to the head of \(Y_V\). This corresponds to attaching an additional module of shape \([Y_V \mid V]\) to the right hand side of \(\Omega E_{2n}\), as discussed above. We can then see that if \(\Omega^{2n} V\) is of the required form, then \(\Omega^{2n+1} V\) is of the required form. A similar approach yields the odd case and so we are done.
\end{proof}

Now that we know the shape of \(\Omega^n V\) for all \(n \geq 0\), we are able to determine the dimensions of \(\H^n(G, V)\) and \(\H^n(G, W)\) for all \(n\) by investigating the modules present in its head.

\begin{proof}[Proof of \cref{3mod4}]
	The result will follow from counting the multiplicity of \(k\) in \(\head \Omega^n V\) by \cref{OmegaCohomology}. For this, we need only count how many times \(k\) appears in the heads of \(E_n\) and \(F_n\), along with the module given in the centre of \(\Omega^n V\). Due to the structure of these modules, we may work modulo 6 for this. We see that \(E_n\) contains \(\ceil*{\frac{n}{6}}\) copies of \(n\) in its head when \(n \equiv 0\), 2, 4 or \(5 \mod 6\) and \(\floor*{\frac{n}{6}}\) otherwise. Similarly, \(F_n\) contains \(\ceil*{\frac{n}{6}}\) copies of \(k\) in its head when \(n \equiv 0\) or \(4 \mod 6\) and \(\floor*{\frac{n}{6}}\) otherwise. Finally, the module given in bold is trivial when \(n\) is odd.

	We prove the case where \(n \equiv 5 \mod 6\) here and leave the remaining cases to the reader. From the above, we have
	\[\ceil*{\frac{n}{6}} + 1 + \floor*{\frac{n}{6}} = 2\ceil*{\frac{n}{6}} = 2 \left( \frac{n}{6} + \frac{1}{6} \right) = \frac{n}{3} + \frac{1}{3} = \ceil*{\frac{n}{3}}\]
	copies of \(k\) in \(\head \Omega^n V\), as required.

	As our decision as to which module is \(V\) and which is \(W\) was arbitrary, we may swap them without loss of generality and thus the result holds for both.
\end{proof}

The following results are direct corollaries of the above analysis, and follow from simply noting other aspects of the structures investigated above.

\begin{cor} \label{1mod4Exts}
	Suppose \(r = 2\) and \(q \equiv 1 \mod 4\). Let \(V\), \(W\) be as above and let \(n > 0\). Then
	\begin{align*}
		\Ext_G^n (W,V) = \Ext_G^n(V,W) &= 0,\\
		\Ext_G^n(W,W) \cong \Ext_G^n(V,V) &\cong \begin{cases}
			0	&	n \equiv 1 \mod 3,\\
			k	&	\text{otherwise}.
		\end{cases}
	\end{align*}
\end{cor}

\begin{proof}
	This follows directly from examining the heads of \(\Omega^n V\) (a periodic module) given in the proof of \cref{1mod4}.
\end{proof}

\begin{cor} \label{3mod4Exts}
	Suppose \(r = 2\) and \(q \equiv 3 \mod 4\). Let \(V\), \(W\) be as above and let \(n > 0\). Then
	\begin{align*}
		\dim \Ext_G^n(V,W) = \dim\Ext_G^n(W,V) &= \ceil*{\frac{n}{3}},\\
		\dim \Ext_G^n(W,W) = \dim \Ext_G^n(V,V) &= \begin{cases}
			\frac{n}{3} + 1 	&	n \equiv 0 \mod 3,\\
			\floor*{\frac{n}{3}}&	n \equiv 1 \mod 3,\\
			\ceil*{\frac{n}{3}}	&	n \equiv 2 \mod 3.
		\end{cases}
	\end{align*}
\end{cor}

\begin{proof}
	We repeat the above process, but counting multiplicities of \(V\) and \(W\) in \(\head \Omega^n V\) instead. As before, \(V\) and \(W\) may be swapped without loss of generality. We first deal with \(\Ext_G^n(V,W)\). The method of counting is identical to above except we may ignore the bold module in \(\Omega^n V\) (see \cref{3mod4structure}) entirely. One can observe as before that \(W\) appears in the head of \(E_n\) with multiplicity \(\floor*{\frac{n}{6}}\) if \(n \equiv 2 \mod 6\) and \(\ceil*{\frac{n}{6}}\) otherwise, and similarly for \(F_n\) with multiplicity \(\ceil*{\frac{n}{6}}\) if \(n \equiv 0\), 2, 4 or \(5 \mod 6\) and \(\floor*{\frac{n}{6}}\) otherwise. We thus get that \(W\) appears with multiplicity \(\floor*{\frac{n}{6}} + \ceil*{\frac{n}{6}} = \ceil*{\frac{n}{3}}\) for \(n \equiv 0\), 1, 2, or \(3 \mod 6\), and \(\ceil*{\frac{n}{6}} + \ceil*{\frac{n}{6}} = \ceil*{\frac{n}{3}}\) for \(n \equiv 4\), \(5 \mod 6\).

	Finally, we deal with \(\Ext_G^n(V,V)\). Again, one may verify that \(V\) appears in the head of \(E_n\) with multiplicity \(\ceil*{\frac{n}{6}}\) if \(n \equiv 0\), 3 or \(5 \mod 6\) and \(\floor*{\frac{n}{6}}\) otherwise, and exactly the same for \(F_n\). Also, \(V\) appears once in the head as the bold module whenever \(n\) is even and not otherwise. Adding these numbers gives the required result.
\end{proof}

\section{Cohomology and extensions in \texorpdfstring{\(\PGL_2(q)\)}{PGL(2,q)}}

We wish to use the results obtained on the cohomology of \(\PSL_2(q)\) to determine cohomology of related groups. There are two obvious directions to go with this: (quasisimple) groups having \(\PSL_2(q)\) as a quotient, and those (almost simple groups) having it as a normal subgroup. We will consider both of these options here, but we will only investigate \(\PGL_2(q)\) and \(\SL_2(q)\). We begin with \(\PGL_2(q)\). Since \(\PSL_2(q) \cong \PGL_2(q)\) for \(q\) even, we need only consider the case where \(q\) is odd in what follows. 

The group \(G\) is not perfect, with \(G' = \PSL_2(q)\), meaning that \(G/G' \cong C_2\) and so \(G\) has a single nontrivial 1-dimensional complex representation, \(\delta\), upon which \(G'\) acts trivially and \(g \in G\setminus G'\) acts as multiplication by \(-1\). The irreducible \(\mathbb{C}G\)-modules are either irreducible \(\mathbb{C}G'\)-modules or their tensor products with \(\delta\), except for the modules corresponding to the two modules of dimension \(\frac{1}{2}(q \pm 1)\). In a manner identical to \(G'\), we obtain irreducible modules \(M_i\) (\(1 < i \leq 2a\)) and \(N_i\) (\(1 \leq i \leq 2b\)) from a maximal torus \(T\) and Singer cycle \(S\) of \(G\), yielding \(q - 2\) distinct irreducible modules of dimension \(q \pm 1\) (where \(a\) and \(b\) are the number of irreducible modules for \(T \cap G'\) and \(S \cap G'\) respectively, up to the action of their respective normalisers). We note that this includes the modules \(M_{2a}\) and \(N_{2a}\) which are irreducible for \(G\) but may split as a direct sum of two modules upon restriction to \(G'\). Including \(k\), \(\delta\), \(V\) (the nontrivial constituent of \(\perm = M_1\)) and \(V \otimes \delta\) we have \(q + 2\) irreducible modules. But \(G\) has \(q + 2\) conjugacy classes and so this must be all of \(\Irr_{\mathbb{C}} G\).

These modules reduce modulo \(r\) in the same way as the irreducible \(\mathbb{C}G'\)-modules.

\subsection*{\texorpdfstring{\(\bm{r}\)}{r} odd}

When \(r\) is odd, the \(\Ext\) groups between irreducibles may easily be determined from what we already know of \(\PSL_2(q)\). One way to do this is to use the Hochschild--Serre spectral sequence for \(\PSL_2(q) \norm \PGL_2(q)\), which immediately collapses into a row and yields the isomorphism \(\H^n(\PGL_2(q), X) \cong \H^n(\PSL_2(q), X)^{G/G'}\) and all that remains is to determine the \(G/G'\)--fixed points of these cohomology groups. Alternatively, one may use the block structures of \(G\) and \(G'\) along with \cite[Lemma 4.3]{FeitBrauerTrees} to see that the blocks of \(G\) are similar to blocks of \(G'\) and use knowledge of \(G\) to show that they have the same Brauer trees, yielding the \(\Ext\) groups.

Using either of these methods, one notes that if an irreducible \(kG\)-module \(X\) is not the only irreducible module in its block (thus either being projective, which one may see from its dimension by \cite[Theorem 1]{BrauerNesbitt} or \cite[Theorem 2.3.2]{DavidGuidebookRepTheory}, or having \(\Ext_G^n(X, X) \cong k\) for all \(n \geq 0\) by \cref{LonelyModule}), then \(X\) lies in one of two blocks of maximal defect with \(\Ext\) groups as follows.

\begin{propn} \label{PGLOdd}
	Let \(V\) be the nontrivial irreducible module in the principal block of \(G\), \(\delta\) be the nontrivial 1-dimensional \(kG\)-module and \(V' \coloneqq V \otimes \delta\). Then, unless \(r \mid q+1\) and the \(r\)-part of \(q+1\) is not 3, for any \(X \in \{k, \delta, V, V'\}\), and \(\{Y, Z\} \in \{\{k, V\}, \{\delta, V'\}\}\),
	\[\Ext_G^n(X, X) \cong \begin{cases}
		0	&	n \equiv 1, \ 2 \mod 4,\\
		k	&	n \equiv 0, \ 3 \mod 4,
	\end{cases} \qquad
	\Ext_G^n(Y, Z) \cong \begin{cases}
		0	&	n \equiv 0, \ 3 \mod 4,\\
		k	&	n \equiv 1, \ 2 \mod 4.
	\end{cases}\]
	If \(r \mid q+1\) and the \(r\)-part of \(q+1\) is not 3, then instead \(\Ext_G^n(V, V) \cong \Ext_G^n(V', V') \cong k\) for all \(n\). Any \(\Ext\) groups not mentioned for these 4 modules are zero.
\end{propn}

\subsection*{\texorpdfstring{\(\bm{r = 2}\)}{r = 2}}

In this case, we may determine almost every \(\Ext\) group with very little effort. If \(X\) is an irreducible module outside of the principal block of \(G\), then \(X\) lies alone in its block and is not projective (\(G\) has no blocks of defect zero), so \(\dim \Ext_G^n(X, X) = 1\) for all \(n \geq 0\) by \cref{LonelyModule}. The \(\Ext\) groups for the nontrivial irreducible module in the principal block are easily found using Shapiro's Lemma, yielding the following.

\begin{propn} \label{PGLChar2NontrivialModules}
	Let \(X\) be any nontrivial irreducible \(kG\)-module not isomorphic to \(\Ind_{G'}^G V\). Then \(\Ext_G^n(X, X) \cong k\) for all \(n\). Now, if \(q \equiv 1 \mod 4\), 
	\[\H^n(G, \Ind_{G'}^G V) \cong \begin{cases}
		k	&	n \equiv 1 \mod 3,\\
		0	&	\text{otherwise}
	\end{cases} \qquad \Ext_G^n(\Ind_{G'}^G V, \Ind_{G'}^G V) \cong \begin{cases}
		0	&	n \equiv 1 \mod 3,\\
		k	&	\text{otherwise},
	\end{cases}\]
	and if \(q \equiv 3 \mod 4\) then \(\dim \H^n(G, \Ind_{G'}^G V) = \ceil*{\frac{n}{3}}\) for all \(n\) and
	\[\dim \Ext_G^n(\Ind_{G'}^G V, \Ind_{G'}^G V) = \begin{cases}
		2 \frac{n}{3} + 1		&	n \equiv 0 \mod 3,\\
		\ceil*{\frac{2n}{3}}	&	n \equiv 1 \mod 3,\\
		2 \ceil*{\frac{n}{3}}	&	n \equiv 2 \mod 3.
	\end{cases}\]
\end{propn}

This leaves only \(\H^n(G, k)\). Recall that \(\PGL_2(q)\) may be regarded as a maximal subgroup of \(H \coloneqq \PSL_2(q^2)\). Using Shapiro's Lemma, we may then change the problem of determining \(\H^n(G, k)\) to a matter of determining the cohomology of some \(\PSL_2(q^2)\)-module about which we know much more. We shall begin by determining the structure of some induced modules which we shall use in the main proof. For the remainder of this section, let \(V\) and \(W\) denote the nontrivial irreducible \(kH\)-modules in the principal block. We will re-use the notation from \eqref{1mod4Projectives} (p.~\pageref{1mod4Projectives}) for submodule structure of projective \(kH\)-modules, and the characters from \cref{1mod4CharacterTable} (p.~\pageref{1mod4CharacterTable}).

\begin{propn} \label{PermutationModulePSLOnPGL}
	Let \(r = 2\) and regard \(\PGL_2(q)\) as a maximal subgroup of \(H \coloneqq \PSL_2(q^2)\). Then
	\[\Ind_{\PGL_2(q)}^{\PSL_2(q^2)} k \cong k \oplus X \oplus Y\]
	where \(\H^n(\PSL_2(q^2), Y) = 0\) for all \(n\) and \(X\) is such that \(\PC(V) \sim [X \mid \rad X]\).
\end{propn}

\begin{proof}
	Taking the notation used in \cref{1mod4CharacterTable}, the permutation character \(\psi\) of \(\PSL_2(q^2)\) acting on \(\PGL_2(q)\) is as below, where \(l_1 \leq \frac{1}{8} (q-1)\), \(l_2 \leq \frac{1}{8} (q+1)\) and \(l_3 < \frac{1}{4}(q^2 - 1)\) is not divisible by \(\frac{1}{2}(q \pm 1)\). The choice as to whether \(\psi(\gamma) = 0\) or \(\psi(\delta) = 0\) is equivalent to choosing which conjugacy class of \(\PGL_2(q) \leq \PSL_2(q^2)\) to act upon.

	\[\begin{array}{lcccccccc} 
			&	1						&	\gamma	&	\delta	&	\alpha^{\frac{1}{4}(q^2 - 1)}	&	\alpha^{\frac{1}{2}(q+1) l_1}	&	\alpha^{\frac{1}{2}(q-1)l_2}	&	\alpha^{l_3}	&	\beta^m \\ \midrule
	\psi	&	\frac{q}{2}(q^2 + 1)	&	q		&	0		&	q								&	\frac{1}{2}(q+1)				&	\frac{1}{2}(q-1)				&	0				&	0 		\\ 
	\end{array}
	\]

	One may then verify by direct computation that
	\[\psi = 1 + \xi_1 + \varphi + \sum_{\substack{q-1 \mid i \\ \text{or } q+1 \mid i}} \chi_i.\]
	By \cite[\nopp VIII(a)]{BurkhardtPSLDecomposition}, we see that the constituents of this, reduced modulo \(2\), lying in the principal block are therefore \(\{k^q, V^{\frac{1}{2}(q+1)}, W^{\frac{1}{2}(q-1)}\}\) (where \(\{a^n\}\) here denotes the \emph{multiset} containing \(a\) with multiplicity \(n\)). By the Green Correspondence, since \(G\) contains the normaliser of a Sylow \(2\)-subgroup of \(H\), we know that \(Z \cong k \oplus X \oplus Y\) for some \(X\) where \(Y\) lies outside the principal block. As we already know the composition factors of \(X\) and \(\PC(V)\) is uniserial, we need only show that \(X\) is indecomposable with head \(V\).

	Note that \(X\) must be self-dual and that, by Shapiro's Lemma, \(\H^1(G, k) \cong k \cong \H^1(H, k) \oplus \H^1(H, X) \cong \H^1(H, X)\). Further, by Frobenius Reciprocity, we know that \(Z^H \cong k\) and so \(X^H = 0\). As such, the socle and, by duality, head of \(X\) contain no trivial summands. But then if \(X \cong Y_1 \oplus Y_2\), we must have that one of \(Y_1\) or \(Y_2\) has zero first cohomology. However, any non-projective module with socle \(W\) or \(V\) must have nonzero first cohomology by \cref{OmegaCohomology} since \(\Ext_H^1(V, V) = \Ext_H^1(V, W) = 0\) (and the same for \(V\) and \(W\) swapped) and so the head of its Heller translate must be trivial. As such, we have that \(X\) is indecomposable with composition factors \(\{k^{q-1}, V^{\frac{1}{2}(q+1)}, W^{\frac{1}{2}(q-1)}\}\), and the only way this may happen with isomorphic, nontrivial head and socle is if \(\soc X \cong \rad X \cong V\) and \(X\) is as required due to the structure of \(\PC(V)\) and \(\PC(W)\).
\end{proof}

\begin{propn} \label{InducedkktoPSL}
	Let \(r = 2\) and regard \(\PGL_2(q)\) as a maximal subgroup of \(H \coloneqq \PSL_2(q^2)\). Then we have a \(kG\)-module \(\Ind_{G'}^G k \sim [k \mid k]\) such that
	\[A \coloneqq \Ind_G^H [k \mid k] \cong [k \mid V \mid k] \oplus \PC(V) \oplus Q\]
	for some module \(Q\) whose every indecomposable summand lies outside the principal block of \(kH\).
\end{propn}

\begin{proof}
	By the previous proposition, the character of this induced module is simply \(2 \psi\) and so the constituents of \(A\) lying in the principal block are \(\{k^{2q}, V^{q+1}, W^{q-1}\}\). By Frobenius Reciprocity, \(\Hom_H(k, A) \cong \Hom_H (k, \Ind_{G'}^H k) \cong \Hom_{G'}(k, k)\) is 1-dimensional and so there is precisely one indecomposable summand of \(\Ind_G^H [k \mid k]\) with trivial socle (and, by duality, one summand with trivial head). 

	Now, either each indecomposable summand of \(A/Q\) is self-dual, or \(A/Q \cong M \oplus M^* \oplus N\) for some \(M\), \(N\). By Shapiro's Lemma, we know that \(\H^n(H, A) \cong \H^n(G', k)\) and in particular \(\H^1(H, A) = 0\). Since \(\Ext_H^1(V, V) = \Ext_H^1(V, W) = 0\), any non-projective indecomposable summand \(Y\) of \(A/Q\) with nontrivial socle will have nonzero first cohomology as \(\head \Omega Y\) must have a trivial summand. As such, \(A/Q\) has no non-projective indecomposable summands with nontrivial socle (and, by duality, head). We thus see that \(A/Q\) has an indecomposable self-dual summand with trivial head and socle, and all other indecomposable summands of \(A/Q\) must be projective. 

	From the composition factors of \(A/Q\), we therefore see that \(A/Q \cong [k \mid V \mid k] \oplus \PC(V)\).
\end{proof}

We are now ready to determine the cohomology of \(k\) using the above two modules.

\begin{propn} \label{PGLChar2TrivialModule}
	For all \(n \geq 0\), we have the following
	\[\dim \H^n(G, k) = \begin{cases}
		\frac{2n}{3} + 1		&	n \equiv 0 \mod 3,\\
		\ceil*{\frac{2n}{3}}	&	n \equiv 1 \mod 3,\\
		2\ceil*{\frac{n}{3}}	&	n \equiv 2 \mod 3.
	\end{cases}\]
	This is equivalent to \(\dim \H^n(G, k) = \dim \H^n(H, k) + \ceil*{\frac{n}{3}}\) where \(H \coloneqq \PSL_2(q^2)\).
\end{propn}

\begin{proof}
	Let \(H \coloneqq \PSL_2(q^2)\) and regard \(G\) as a maximal subgroup of \(H\). From \cref{PermutationModulePSLOnPGL} and Shapiro's Lemma, we see that \(\H^n(G, k) \cong \H^n(H, k) \oplus \H^n(H, X)\). It is thus sufficient to show that \(\dim \H^n(H, X) = \ceil*{\frac{n}{3}}\). Since \(M \coloneqq \rad X \cong \Omega X\), this is equivalent to showing that \(\dim \H^n(H, M) = \ceil*{\frac{n-1}{3}}\). First, we recall that \(\H^n(H, V)\) is known by \cref{1mod4} to be 1-dimensional if \(n \equiv 1 \mod 3\) and zero otherwise. Now, \(\Omega V \sim [M \mid M]\), and so \(\H^n(H, [M \mid M])\) is 1-dimensional for \(n \equiv 2 \mod 3\) and zero otherwise by \cref{OmegaCohomology}, and by using this information in the long exact sequence in cohomology corresponding to the extension \(0 \to M \to \Omega V \to M \to 0\) we see that, for \(n \geq 0\), \(\H^{3n}(H, M) \cong \H^{3n+1}(H, M)\) and the sequence
	\[0 \to \H^{3n+1}(H, M) \to \H^{3n+2}(H, M) \to k \to \H^{3n+2}(H, M) \to \H^{3n+3}(H, M) \to 0 \]
	is exact. This tells us that, for each \(n\), we have two possibilities: either
	\[\dim \H^{3n+1}(H, M) = \dim \H^{3n+2}(H, M) = \dim \H^{3n+3}(H, M) + 1 \label{mess1} \tag{1}\]
	or
	\[\dim \H^{3n+3}(H, M) = \dim \H^{3n+2}(H, M) = \dim \H^{3n+1}(H, M) + 1. \label{mess2} \tag{2}\]
	We show that the latter always holds. It is worth noting above that if the latter (and larger) case always holds, we have \(\dim \H^n(H, M) = \ceil{\frac{n-1}{3}}\) for any \(n\).

	By Shapiro's Lemma and \cref{InducedkktoPSL}, we see that \(\H^n(H, [k \mid V \mid k]) \cong \H^n(G', k)\) (and similarly for \([k \mid W \mid k]\), as our choice of \(V\) was determined by the particular embedding of \(G\) into \(H\)) and from the long exact sequence in cohomology corresponding to the extension \(0 \to [k \mid W \mid k] \to [V \mid k \mid W \mid k] \to V \to 0\), we see that for \(n \geq 0\), \(\H^{3n}(H, [V \mid k \mid W \mid k]) \cong \H^{3n}(H, [k \mid W \mid k])\) has dimension \(\frac{3n}{3} + 1\). Now, we may construct a quotient (or submodule) \(Y \sim [k \mid \heart(M) \oplus [V \mid k \mid W] \mid k]\) of \(\PC(k)\) and observe that \(\Omega^2 Y \cong Y\), with the heads of both \(Y\) and \(\Omega Y\) being trivial so that \(\H^n(H, Y)\) is 1-dimensional for all \(n\). Finally, from the long exact sequence in cohomology corresponding to the extension \(0 \to [V \mid k \mid W \mid k] \to Y \to M \to 0\), we see that, for \(n \geq 0\), 
	\[\frac{3n}{3} + 1 = \dim \H^{3n}(H, [V \mid k \mid W \mid k]) \leq \dim \H^{3n-1}(H, M) + 1 \leq \frac{3n}{3} + 1.\]
	Thus \(\dim \H^{3n-1}(H, M) = \frac{3n}{3}\) for any \(n \geq 1\). This is only possible if \eqref{mess2} holds for each \(n \geq 0\), and so we are done.
\end{proof}

\section{Cohomology and extensions in \texorpdfstring{\(\SL_2(q)\)}{SL(2,q)}}

Another group whose representation theory is very closely linked to that of \(H \coloneqq \PSL_2(q)\) is its covering group, \(G \coloneqq \SL_2(q)\). For this section, suppose \(q\) is odd so that \(G \ncong \PSL_2(q)\).

For the remainder of this section, let \(\tilde{T}\) and \(\tilde{S}\) denote some fixed preimage under the quotient map of a maximal torus \(T\) or Singer cycle \(S\) of \(H\), respectively. We will also denote the irreducible \(kG\)-modules in the same way as irreducible \(kH\)-modules as the constructions used for \(H\) will still work for \(G\) (see \cite[\textsection 38]{LarryA}).

In the case where \(r\) is odd, there is essentially nothing to do. The Hochschild--Serre spectral sequence corresponding to \(Z(G) \norm G\) collapses to yield \(\H^n(G, V) \cong \H^n(G/Z(G), V^{Z(G)})\) and this can easily be combined with \cref{PSL2MinusOneSummary,PSL2PlusOneSummary} to give all \(\Ext\) groups. One could also quickly compute this directly using Chapter 9 and Theorem 7.1.1 of \cite{SL2Representations} along with \cref{LonelyModule,PSL2OtherExtsCase2}.

\subsection*{\texorpdfstring{\(\bm{r = 2}\)}{r = 2}}

As is often the case, the situation where \(r = 2\) requires the most work since the Sylow \(2\)-subgroups of \(G\) are generalised quaternion (so, not cyclic) and \(G\) has blocks containing both faithful and non-faithful modules. Further, in \(\SL_2(q)\), both \(\tilde{T}\) and \(\tilde{S}\) have even order and so, since the defect groups of blocks of non-maximal defect are the \(2\)-parts of \(\tilde{T}\) or \(\tilde{S}\) which are never trivial, there are no projective irreducible \(kG\)-modules here. However, since the defect groups of blocks of non-maximal defect are contained in \(\tilde{T}\) or \(\tilde{S}\), in particular they are cyclic and so we may still deal with these blocks easily using \cref{LonelyModule}.

Since the Sylow \(2\)-subgroups of \(G\) are generalised quaternion, by \cite[1]{ErdmannQuaternion1}, all simple modules in blocks of maximal defect (in this case, the principal block) are periodic of period dividing \(4\). This simple fact turns out to be a great help, since if one can show that \(\Omega^2 V \ncong V\) for some \(V\) in a block with generalised quaternion defect group, we know that \(V\) has period 4 and so \(\Omega^3 V \cong \PC(V)/V\) and \(\Omega^4 V \cong V\). In particular, we need only calculate \(\Omega^2 V\) (as \(\Omega V \cong \rad \PC(V)\)) in order to determine \(\Ext_G^n(V, W)\) for any irreducible \(W\).

The following is immediate from \cite[Theorem 7.1.1]{SL2Representations} and \cref{LonelyModule}.

\begin{propn} \label{SL2EvenMAndN}
	Suppose that \(Y\) is any of the \(M_i\) or \(N_i\) except \(M_1 = \perm\). Then \(\Ext_G^n(Y, Y) \cong k\) for all \(n\).
\end{propn}

The bulk of this section will consist of the proofs of the following two propositions, dealing with the only remaining cases.

\begin{propn} \label{SL2EvenPrincipalBlock1mod4}
	Let \(r = 2\) and \(q \equiv 1 \mod 4\). Then the following hold.
	\[\H^n(G, k) \cong \begin{cases}
		0	&	n \equiv 1, \ 2 \mod 4,\\
		k	&	n \equiv 0, \ 3 \mod 4,
	\end{cases} \qquad
	\H^n(G, V) \cong \begin{cases}
		0	&	n \equiv 0, \ 3 \mod 4,\\
		k	&	n \equiv 1, \ 2 \mod 4,
	\end{cases}\]
	and
	\[\Ext_G^n(V, V) \cong \begin{cases}
		0	&	n \equiv 1, \ 2 \mod 4,\\
		k	&	n \equiv 0, \ 3 \mod 4.
	\end{cases}\]
	Further, \(\Ext_G^n(V, W) = 0\) for all \(n\) and without loss of generality we may swap \(V\) and \(W\) in all of the above to obtain the corresponding results for \(W\).
\end{propn}

\begin{propn} \label{SL2EvenPrincipalBlock3mod4}
	Let \(r = 2\) and \(q \equiv 3 \mod 4\). Then the following hold.
	\[\H^n(G, k) \cong \begin{cases}
		0	&	n \equiv 1, \ 2 \mod 4,\\
		k	&	n \equiv 0, \ 3 \mod 4,
	\end{cases} \qquad
	\H^n(G, V) \cong \begin{cases}
		0	&	n \equiv 0, \ 3 \mod 4,\\
		k	&	n \equiv 1, \ 2 \mod 4,
	\end{cases}\]
	and
	\[\Ext_G^n(V, V)\cong \begin{cases}
		0	&	n \equiv 1, \ 2 \mod 4,\\
		k	&	n \equiv 0, \ 3 \mod 4,
	\end{cases} \qquad 
	\Ext_G^n(V, W) \cong \begin{cases}
		0	&	n \equiv 0, \ 3 \mod 4,\\
		k	&	n \equiv 1, \ 2 \mod 4,
	\end{cases}\]
	where without loss of generality we may swap \(V\) and \(W\) in all of the above to obtain the corresponding results for \(W\).
\end{propn}

To prove the above, we use knowledge of the structure of the projective \(kG\)-modules lying in the principal block. These are given by Erdmann (indicating the extensions used) in \cite[Appendix]{ErdmannQuaternion2}, where the association of these projective covers with \(\SL_2(q)\) is given in (7.9)(a) and (c) of the same article. We summarise the information we require via the two below results.

\begin{propncite}[{\cite[Appendix III]{ErdmannQuaternion2}}] \label{SL21mod4Projectives}
	Let \(r = 2\) and \(q \equiv 1 \mod 4\). Then the projective indecomposable modules in the principal block of \(kG\) are as below, where \(Y_1\) and \(Y_2\) are as in the proof of \cref{1mod4} and \(\PC(k)\) is given only as radical factors.
	\[
	\begin{array}{ccc}
			&	V			&			\\
			&	k			&	 		\\
			&				&	W 		\\
			&				&	k 		\\
			&				&	V		\\
		V	&	\bigoplus	&	k		\\
			&				& 	\vdots 	\\
			&				&	V		\\
			&				&	k		\\
			& 				&	W 		\\
			&	k			&	 		\\
			&	V 			&
	\end{array}
	\qquad \qquad \qquad
	\begin{array}{ccc}
				&	W			&			\\
				&	k			&	 		\\
		V		&				&	 		\\
		k		&				&	 		\\
		W		&				&			\\
		k		&	\bigoplus	&	W		\\
		\vdots	&				& 		 	\\
		W		&				&			\\
		k		&				&			\\
		V		& 				&	 		\\
				&	k			&	 		\\
				&	W 			&
	\end{array}
	\qquad \qquad \qquad
	\begin{array}{ccc}
				&	k 			& 			\\
		Y_1		&	\bigoplus	&	Y_2		\\
		k		&	\bigoplus	&	k		\\
		Y_1		&	\bigoplus	&	Y_2		\\
				&	k 			&			
	\end{array}
\]
The composition factors of the above are as follows, where \(\{a^n\}\) denotes the \emph{multiset} containing \(a\) with multiplicity \(n\). \(\PC(k)\) has composition factors \(\{k^{4m+2}, V^{4m}, W^{4m}\}\), \(\PC(V)\) has composition factors \(\{k^{4m}, V^{2m+2}, W^{2m}\}\) and \(\PC(W)\) has composition factors \(\{k^{4m}, V^{2m}, W^{2m+2}\}\).
\end{propncite}

\begin{propncite}[{\cite[Appendix VI]{ErdmannQuaternion2}}] \label{SL23mod4Projectives}
	Let \(r = 2\) and \(q \equiv 3 \mod 4\). Then the projective indecomposable modules in the principal block of \(kG\) have radical factors as below, where \(Y_V\) and \(Y_W\) are as seen after the statement of \cref{3mod4} and the factors \(k\), \(V\) and \(k\) (respectively \(k\), \(W\), \(k\)) all appear alongside factors of the upper \(Y_V\) (respectively \(Y_W\)).
	\[
	\begin{array}{ccc}
			&	V			&			\\
		k	&	\bigoplus	&	Y_V		\\
		V	&	\bigoplus	&	V		\\
		k	&	\bigoplus	&	Y_V		\\
			&	V			&					
	\end{array}
	\qquad \qquad \qquad
	\begin{array}{ccc}
			&	W			&			\\
		Y_W	&	\bigoplus	&	k		\\
		W	&	\bigoplus	&	W		\\
		Y_W	&	\bigoplus	&	k		\\
			&	W			&					
	\end{array}
	\qquad \qquad \qquad
	\begin{array}{ccc}
		&	k 			& 		\\
	V 	& 	\bigoplus 	&	W 	\\
	k	& 	\bigoplus	&	k	\\
	V	&	\bigoplus	&	W	\\
		&	k			&
	\end{array}
\]
The composition factors of the above are as follows, where \(\{a^n\}\) denotes the \emph{multiset} containing \(a\) with multiplicity \(n\). \(\PC(k)\) has composition factors \(\{k^{2}, V^{4}, W^{4}\}\), \(\PC(V)\) has composition factors \(\{k^{2}, V^{2m+2}, W^{2m}\}\) and \(\PC(W)\) has composition factors \(\{k^{2}, V^{2m}, W^{2m+2}\}\).
\end{propncite}

We will prove \cref{SL2EvenPrincipalBlock1mod4} by investigating the structure of \(\Omega^n k\) and \(\Omega^n V\) for \(n \leq 4\), and as such the result is best proven in two parts.

\begin{lemma} \label{SL21mod4Cohomology}
	Let \(q \equiv 1 \mod 4\). Then the cohomology of \(k\), \(V\) and \(W\) is as in \cref{SL2EvenPrincipalBlock1mod4}.
\end{lemma}

\begin{proof}
	We proceed as usual. The shape of \(\Omega k\) is clear from the structure of \(\PC(k)\). Looking at the composition factors of \(\Omega k\), we can see that the composition factors of \(\Omega^2 k\) must be two copies of \(V\), two copies of \(W\) and a single copy of \(k\). Note that \(\Omega^2 k\) must be a submodule of \(\PC(V) \oplus \PC(W)\) with these composition factors, and the only possible way of taking such a submodule is of shape \([V \oplus W \mid k \mid V \oplus W]\) since \(\soc^3(\PC(V) \oplus \PC(W))\) is of shape \([V \oplus W \oplus V \oplus W \mid k \oplus k \mid V \oplus W]\). Since \(\Omega^2 k\) is not simple and \(k\) has period dividing 4, \(k\) must have period 4 and so \(\Omega^3 k \cong \Omega^{-1} k \cong \PC(k)/k\). The result then follows from the fact that \(\head \Omega^0 k \cong k\), \(\head \Omega k \cong V \oplus W\), \(\head \Omega^2 k \cong V \oplus W\) and \(\head \Omega^3 k \cong k\).
\end{proof}

\begin{lemma} \label{SL21mod4ExtV}
	Let \(q \equiv 1 \mod 4\). Then the \(\Ext\) groups not described in \cref{SL21mod4Cohomology} are as in \cref{SL2EvenPrincipalBlock1mod4}.
\end{lemma}

\begin{proof}
	The shape of \(\Omega V\) is clear from the shape of \(\PC(V)\). We now note that \(\Omega V\) must be covered by \(\PC(k)\), and is in fact of shape \([k \mid V \oplus Y_2 \mid k \mid Y_2]\). The kernel of this covering thus has shape \(\Omega^2 V \sim [\rad Y_1 \mid k \mid Y_1 \mid k] \sim [k \mid Y_1 \mid k \mid Y_1 / \soc Y_1]\) and is again covered by \(\PC(k)\). Since \(\Omega^2 V\) is not simple and we know that \(V\) has period dividing 4, we see that \(V\) has period 4 and so \(\Omega^3 V \sim \PC(V)/V\) and \(\Omega^4 V \cong V\). The result then follows from the fact that \(\head \Omega^0 V \cong V\), \(\head \Omega V \cong k\), \(\head \Omega^2 V \cong k\) and \(\head \Omega^3 V \cong V\).
\end{proof}

As with the previous case, \cref{SL2EvenPrincipalBlock3mod4} is best proven in two parts.

\begin{lemma} \label{SL23mod4Cohomology}
	Let \(q \equiv 3 \mod 4\). Then the cohomology of \(k\), \(V\) and \(W\) is as in \cref{SL2EvenPrincipalBlock3mod4}.
\end{lemma}

\begin{proof}
	We proceed as usual. The shape of \(\Omega k\) is clear from the structure of \(\PC(k)\). Next, we know that the composition factors of \(\Omega^2 k\) are those of \(\PC(\Omega k)\) not lying in \(\Omega k\), namely \(\{k^{4-3}, V^{4m+2-2}, W^{4m+2-2}\} = \{k^1, V^{4m}, W^{4m}\}\). In particular, \(\Omega^2 k\) is not simple, and since \(V\) has period dividing 4 its period must therefore be 4. This tells us that \(\Omega^3 V \cong \Omega^{-1} V \cong \PC(k)/k\) with composition factors \(\{k^3, V^2, W^2\}\). But we also know that the composition factors of \(\Omega^3 k\) are those of \(\PC(\Omega^2 k)\) which are not in \(\Omega^2 k\). From this, one deduces that the composition factors of \(\PC(\Omega^2 V)\) are the union of the composition factors for \(\Omega^2 k\) and \(\Omega^3 k\), namely \(\{k^4, V^{4m+2}, W^{4m+2}\}\). If we suppose that \(\PC(\Omega^2 k) \cong \PC(k)^{\oplus a} \oplus \PC(V)^{\oplus b} \oplus \PC(W)^{\oplus c}\), then we get the following set of simultaneous equations, which are easily solved to find that \(b = c = 1\) and \(a = 0\).
	\begin{align*}
		4a + 2b + 2c &= 4,\\
		2a + (2m+2)b + 2mc &= 4m+2,\\
		2a + 2mb + (2m+2)c &= 4m+2.
	\end{align*}
	This tells us that \(\PC(\Omega^2 k) \cong \PC(V) \oplus \PC(W)\) and so \(\head \Omega^2 k \cong V \oplus W\). The result then follows from the fact that \(\head \Omega^0 k \cong k\), \(\head \Omega^k \cong V \oplus W\), \(\head \Omega^2 k \cong V \oplus W\) and \(\head \Omega^3 k \cong k\).
\end{proof}

\begin{lemma} \label{SL23mod4ExtV}
	Let \(q \equiv 3 \mod 4\). Then the \(\Ext\) groups not described in \cref{SL23mod4Cohomology} are as in \cref{SL2EvenPrincipalBlock3mod4}.
\end{lemma}

\begin{proof}
	We proceed as in \cref{SL23mod4Cohomology}. The shape of \(\Omega V\) is clear from the structure of \(\PC(V)\). As above, we know that the composition factors of \(\Omega^2 V\) are \(\{ k^{4 + 2 - 2}, V^{2 + 2m -2m -1}, W^{2 + 2m + 2 - 2m}\} = \{k^4, V^1, W^4\}\). In particular, \(\Omega^2 V\) is not simple, and so \(V\) must have period 4. This tells us that \(\Omega^3 V \cong \Omega^{-1} V \cong \PC(V)/V\) with composition factors \(\{k^2, V^{2m+1}, W^{2m}\}\). But proceeding as in \cref{SL23mod4Cohomology} we see that the composition factors of \(\PC(\Omega^2 V)\) are \(\{k^6, V^{2m+2}, W^{2m+4}\}\). If we suppose that \(\PC(\Omega^2 V) \cong \PC(k)^{\oplus a} \oplus \PC(V)^{\oplus b} \oplus \PC(W)^{\oplus c}\), then we get the following set of simultaneous equations, which are easily solved to find that \(a = c = 1\) and \(b = 0\).
	\begin{align*}
		4a + 2b + 2c &= 6,\\
		2a + (2m+2)b + 2mc &= 2m+2,\\
		2a + 2mb + (2m+2)c &= 2m+4.
	\end{align*}
	This tells us that \(\head \Omega^2 V \cong k \oplus W\). The result then follows from the fact that \(\head \Omega^0 V \cong V\), \(\head \Omega V \cong k \oplus W\), \(\head \Omega^2 V \cong k \oplus W\) and \(\head \Omega^3 V \cong V\).
\end{proof}

\newpage

\printbibliography

\end{document}